\documentclass[11pt, a4paper]{amsart}
\usepackage{amsmath, amsthm, amssymb, amsfonts, mathtools}
\usepackage{graphicx}
\usepackage{geometry}
\usepackage{a4wide}
\usepackage[latin1]{inputenc}

\DeclareMathOperator{\arccot}{arccot}

\DeclareMathOperator{\CT}{CT}
\DeclareMathOperator{\Res}{Res}

\DeclareMathOperator{\spann}{span}
\DeclareMathOperator{\sgn}{sgn}
\DeclareMathOperator{\reg}{reg}

\DeclareMathOperator{\SL}{SL}
\DeclareMathOperator{\Mp}{Mp}

\DeclareMathOperator{\Z}{\mathbb{Z}}
\DeclareMathOperator{\R}{\mathbb{R}}
\DeclareMathOperator{\C}{\mathbb{C}}
\DeclareMathOperator{\Q}{\mathbb{Q}}

\renewcommand{\H}{\mathbb{H}}
\renewcommand{\Re}{\text{Re}}
\DeclareMathOperator{\tr}{tr}

\DeclareMathOperator{\calQ}{\mathcal{Q}}

\DeclareMathOperator{\e}{\mathfrak{e}}

	\newtheorem{Satz}{Satz}[section]
	\newtheorem{theorem}[Satz]{Theorem}
	\newtheorem{lemma}[Satz]{Lemma}
	\newtheorem{proposition}[Satz]{Proposition} 
	\newtheorem{corollary}[Satz]{Corollary}
	
	\theoremstyle{definition} 
	
	\newtheorem{example}[Satz]{Example}
	\newtheorem{remark}[Satz]{Remark}
\date{\today}
\title[Borcherds lifts of harmonic Maass forms]{Borcherds lifts of harmonic Maass forms and modular integrals}

\address{Mathematical Institute, University of Cologne, Weyertal 86-90, D--50931 Cologne, Germany}

\author{Markus Schwagenscheidt}
\email{mschwage@math.uni-koeln.de}
\thanks{This work is a shortened version of a chapter of my PhD thesis. I am indebted to my advisor Jan Bruinier for proposing the topic of this work to me and for many enlightening discussions. I also thank Kathrin Bringmann, Stephan Ehlen, and Yingkun Li for several helpful discussions on the topic. During the preparation of this work, I was partially supported by the DFG Research Unit FOR 1920 \lq Symmetry, Geometry and Arithmetic\rq, by the LOEWE Reseach Unit USAG, and by the SFB-TRR 191 \lq Symplectic Structures in Geometry, Algebra and Dynamics\rq, funded by the DFG
}

\begin{document}

\begin{abstract}
We extend Borcherds' singular theta lift in signature $(1,2)$ to harmonic Maass forms of weight $1/2$ whose non-holomorphic part is allowed to be of exponential growth at $i\infty$. We determine the singularities of the lift and compute its Fourier expansion. It turns out that the lift is continuous but not differentiable along certain geodesics in the upper half-plane corresponding to the non-holomorphic principal part of the input. As an application, we obtain a generalization to higher level of the weight $2$ modular integral of Duke, Imamoglu and T\'oth. Further, we construct automorphic products associated to harmonic Maass forms.
\end{abstract}

\maketitle

\section{Introduction}

In \cite{borcherds95}, Borcherds constructed a regularized theta lift which maps weakly holomorphic modular forms of weight $1/2$ to real analytic modular functions with logarithmic singularities at CM points. His results were generalized by Bruinier and Ono \cite{bruinieronoheegnerdivisors} to twisted lifts of harmonic Maass forms which map to cusp forms under the $\xi$-operator. In the present work, we extend the twisted Borcherds lift to general harmonic Maass forms (which may map to weakly holomorphic modular forms under the $\xi$-operator). By taking the derivative of the Borcherds lift of a suitable harmonic Maass form of weight $1/2$, we obtain modular integrals of weight $2$ with rational period functions. Their Fourier coefficients are given by twisted traces of geodesic cycle integrals of harmonic Maass forms of weight $0$. They generalize the modular integral of Duke, Imamoglu and T\'oth \cite{dit} to higher level. In the introduction, we restrict to modular forms for the full modular group $\Gamma = \SL_{2}(\Z)$ for simplicity, but in the body of the work we treat modular forms of arbitrary level $\Gamma_{0}(N)$ by using the language of vector valued modular forms for the Weil representation. Let us now describe our results in more detail.

\subsection{The Borcherds lift of a harmonic Maass form} Recall from \cite{bruinierfunke04} that a harmonic Maass form of weight $1/2$ for $\Gamma_{0}(4)$ is a smooth function $f: \H \to \C$ which is annihilated by the invariant Laplace operator $\Delta_{1/2}$, transforms like a modular form of weight $1/2$ for $\Gamma_{0}(4)$, and is at most of linear exponential growth at the cusps of $\Gamma_{0}(4)$. Such a form can be written as a sum $f = f^{+} + f^{-}$ with a holomorphic part $f^{+}$ and a non-holomorphic part $f^{-}$ with Fourier expansions of the shape
\begin{align*}
f^{+}(\tau) &= \sum_{\substack{D \in \Z}}c_{f}^{+}(D)e(D\tau), \\
f^{-}(\tau) &=c_{f}^{-}(0)\sqrt{v} + \sum_{ D < 0}c_{f}^{-}(D)\sqrt{v}\beta_{1/2}(4\pi |D|v)e(D\tau) + \sum_{D > 0}c_{f}^{-}(D)\sqrt{v}\beta_{1/2}^{c}(-4\pi Dv)e(D\tau),
\end{align*}
with $\tau = u + iv \in \H$, $e(x) = e^{2\pi i x}$ for $x \in \C$, coefficients $c_{f}^{\pm}(D) \in \C$, and 
\[
\beta_{1/2}(s) = \int_{1}^{\infty}e^{-st}t^{-1/2}dt, \qquad \beta_{1/2}^{c}(s) = \int_{0}^{1}e^{-st}t^{-1/2}dt.
\]
We let $H_{1/2}$ denote the space of harmonic Maass forms of weight $1/2$ which satisfy the Kohnen plus space condition, which means that the Fourier expansion is supported on indices $D \equiv 0,1 (4)$. The antilinear differential operator
\[
\xi_{1/2}f(\tau) = 2iv^{1/2}\overline{\frac{\partial}{\partial \bar{\tau}}f(\tau)}
\]
maps a harmonic Maass form $f \in H_{1/2}$ of weight $1/2$ to a weakly holomorphic modular form of weight $3/2$. We let $H_{1/2}^{+}$ be the subspace of $H_{1/2}$ consisting of forms which map to cusp forms under $\xi_{1/2}$, and we let $M_{1/2}^{!}$ be the subspace of weakly holomorphic modular forms.

Let $\Delta \in \Z$ be a fundamental discriminant. For simplicity, we assume $\Delta > 1$ in the introduction. Following \cite{borcherds95}, we define the Borcherds lift $\Phi_{\Delta}(z,f)$ of a harmonic Maass form $f \in H_{1/2}$ by the regularized integral
\[
\Phi_{\Delta}(f,z) = \CT_{s = 0}\left[\lim_{T \to \infty}\int_{\mathcal{F}_{T}(4)} f(\tau)\overline{\Theta_{\Delta}(\tau,z)}v^{1/2-s} \frac{du \, dv}{v^{2}}\right],
\]
where $\Theta_{\Delta}(\tau,z)$ is a twisted Siegel theta function which transforms in $\tau$ like a modular form of weight $1/2$ for $\Gamma_{0}(4)$ and is invariant in $z$ under $\Gamma$, $\mathcal{F}_{T}(4)$ denotes a suitably truncated fundamental domain for $\Gamma_{0}(4) \backslash \H$, and $\CT_{s = 0}F(s)$ denotes the constant term in the Laurent expansion at $s =0$ of a function $F(s)$ which is meromorphic near $s = 0$. Borcherds \cite{borcherds} proved that for $\Delta = 1$ and a weakly holomorphic modular form $f\in M_{1/2}^{!}$ the regularized theta lift $\Phi_{\Delta}(f,z)$ defines a $\Gamma$-invariant real analytic function with logarithmic singularities at certain CM points in $\H$, which are determined by the principal part of $f$, i.e., by the coefficients $c_{f}^{+}(D)$ with $D < 0$. Bruinier and Ono \cite{bruinieronoheegnerdivisors} showed that this result remains true for twisted Borcherds lifts of harmonic Maass forms $f \in H_{1/2}$ which map to cusp forms under the $\xi$-operator, which means that $c_{f}^{-}(0) = 0$ and $c_{f}^{-}(D) = 0$ for $D > 0$. One of the main aims of the present work is to generalize the Borcherds lift $\Phi_{\Delta}(z,f)$ to the full space $H_{1/2}$.

For a discriminant $D$ we let let $\calQ_{D}$ be the set of integral binary quadratic forms $Q = [a,b,c]$ of discriminant $D = b^{2}-4ac$. For $D < 0$ and $Q \in \calQ_{D}$ there is an associated CM (or Heegner) point $z_{Q}\in\H$
	which is characterized by $Q(z_{Q},1) = 0$.  For $D > 0$ there is an associated geodesic in $\H$ given by
	\[
	c_{Q} = \{z \in \H: a|z|^{2} + bx + c = 0\},
	\]
	with $z = x+iy \in \H$. We let $H_{\Delta}^{+}(f)$ be the set of all CM points $z_{Q}$ corresponding to quadratic forms $Q \in \calQ_{\Delta D}$ with $D < 0$ such that $c_{f}^{+}(D) \neq 0$, and we let $H_{\Delta}^{-}(f)$ be the union of all geodesics $c_{Q}$ corresponding to quadratic forms $Q \in \calQ_{\Delta D}$ with $D > 0$ such that $c_{f}^{-}(D) \neq 0$. We obtain the following extension of the Borcherds lift on the full space $H_{1/2}$.

\begin{theorem}\label{theorem Borcherds lift introduction}
Let $\Delta > 1$ be a fundamental discriminant. For $f \in H_{1/2}$ the Borcherds lift $\Phi_{\Delta}(f,z)$ defines a $\Gamma$-invariant harmonic function on $\H \setminus (H^{+}_{\Delta}(f) \cup H^{-}_{\Delta}(f))$. It has \lq logarithmic singularities\rq \ at the CM points in $H^{+}_{\Delta}(f)$ and \lq arcsin singularities\rq \ along the geodesics in $H_{\Delta}^{-}(f)$. More precisely, this means that for $z_{0} \in H_{\Delta}^{+}(f) \cup H_{\Delta}^{-}(f)$ the function
\begin{align*}
\Phi_{\Delta}(f,z) &-\sum_{D < 0}c_{f}^{+}(D)\!\!\!\sum_{\substack{Q=[a,b,c] \in \calQ_{\Delta D} \\ z_{0} = z_{Q}}}\chi_{\Delta}(Q)\log|az^{2} + bz + c| \\
		&+\sum_{D > 0}\frac{c_{f}^{-}(D)}{\sqrt{D}}\!\!\!\sum_{\substack{Q= [a,b,c] \in \calQ_{\Delta D} \\ z_{0} \in c_{Q}}}\chi_{\Delta}(Q)\arcsin\left(\frac{1}{\sqrt{1+\frac{1}{\Delta Dy^{2}}(a|z|^{2} + bx + c)^{2}}}\right).
\end{align*}
can be continued to a real analytic function near $z_{0}$. Here $\chi_{\Delta}$ is the usual genus character. Note that all the above sums are finite.
\end{theorem}

We refer the reader to Theorem \ref{theorem Borcherds lift} for the general result. 
\begin{remark}The logarithmic singularities imply that the Borcherds lift blows up at the Heegner points $z_{Q} \in H_{\Delta}^{+}(f)$, and the arcsin singularities show that it is continuous but not differentiable at points on the geodesics $c_{Q} \subset H_{\Delta}^{-}(f)$. 
\end{remark}

\subsection{The Fourier expansion of the Borcherds lift} Using Maass-Poincar\'e series one can always write a harmonic Maass form $f \in H_{1/2}$ as $f = f_{1} + f_{2}$ where $f_{1},f_{2} \in H_{1/2}$ satisfy  $c_{f_{1}}^{+}(D) = 0$ for all $D < 0$ and $c_{f_{2}}^{-}(D) = 0$ for all $D \geq 0$. In particular, $f_{2}$ maps to a cusp form under the $\xi$-operator, and since the Borcherds lift of such harmonic Maass forms has already been investigated by Bruinier and Ono \cite{bruinieronoheegnerdivisors}, we assume from now on that $c_{f}^{+}(D) = 0$ for all $D < 0$. In this case, the Borcherds lift $\Phi_{\Delta}(f,z)$ only has singularities along the geodesics in $H_{\Delta}^{-}(f)$. Furthermore, the Fourier expansion of $\Phi_{\Delta}(f,z)$ can be stated as follows.

\begin{proposition}	\label{proposition Fourier expansion introduction}
		Let $\Delta > 1$ be a fundamental discriminant and let $f \in H_{1/2}$ such that $c_{f}^{+}(D) = 0$ for all $D < 0$. Then for $z \in \H \setminus H_{\Delta}^{-}(f)$ the Borcherds lift of $f$ has the Fourier expansion
		\begin{align*}
		\Phi_{\Delta}(f,z) &=
			-4\sum_{m=1}^{\infty}c_{f}^{+}(\Delta m^{2})\sum_{b (\Delta)}\left( \frac{\Delta}{b}\right)\log|1-e(mz + b/\Delta)| \\
			&\quad+\sqrt{\Delta}\, L_{\Delta}(1) \left(2c_{f}^{+}(0) + y\, c_{f}^{-}(0)\right)  \\
			&\quad-4\sum_{D > 0}\frac{c_{f}^{-}(D)}{\sqrt{D}}\sum_{\substack{Q \in \mathcal{Q}_{\Delta D} \\ a > 0}}\chi_{\Delta}(Q)\mathbf{1}_{Q}(z)\left(\arctan\left(\frac{y\sqrt{\Delta D}}{a|z|^{2} + bx + c}\right) + \frac{\pi}{2}\right),
		\end{align*}
			where $\mathbf{1}_{Q}(z)$ denotes the characteristic function of the bounded component of $\H \setminus c_{Q}$, and $L_{\Delta}(s) = \sum_{n \geq 1}\left( \frac{\Delta}{n}\right)n^{-s}$ for $\Re(s)>1$ is a Dirichlet $L$-function.
	\end{proposition}

For the general result, see Proposition \ref{proposition Fourier expansion}.
\begin{remark} 
	\begin{enumerate}
		\item For $Q \in \calQ_{\Delta D}$ with $a > 0$ the corresponding geodesic $c_{Q}$ is a semi-circle centered at the real line which divides $\H$ into a bounded and an unbounded connected component, so the characteristic function $\mathbf{1}_{Q}$ makes sense.
		\item The sum over $D$ in the third line is finite since $f$ has a finite principal part. The sum over $Q \in \mathcal{Q}_{\Delta D}$ is locally finite since each point $z \in \H$ lies in the bounded component of $\H \setminus c_{Q}$ for finitely many geodesics $c_{Q} \in \calQ_{\Delta D}$, and it vanishes for $y \gg 0$ large enough since the imaginary parts of points lying on geodesics $c_{Q}$ for $Q \in \calQ_{\Delta D}$ are bounded by $\sqrt{\Delta D}$.
		\item We have $z \in c_{Q}$ for $Q = [a,b,c]$ if and only if $a|z|^{2} + bx + c = 0$. Further, for $a > 0$ a point $z \in \H$ lies in the inside of the bounded component of $\H\setminus c_{Q}$ if and only if $a|z|^{2} + bx + c < 0$. Since $\lim_{x \to -\infty}\arctan(x) = -\frac{\pi}{2}$, we see from the Fourier expansion that $\Phi_{\Delta}(f,z)$ is continuous. However, computing the derivative of the above expansion for $z \in \H \setminus H_{\Delta}^{-}(f)$ shows that the third line is not differentiable at points $z \in H_{\Delta}^{-}(f)$. More precisely, the derivative of $\Phi_{\Delta}(f,z)$ has jumps along the geodesics in $H_{\Delta}^{-}(f)$.
	\end{enumerate}
\end{remark}

\subsection{The derivative of the Borcherds lift} We apply the (derivative of the) Borcherds lift to certain interesting harmonic Maass forms of weight $1/2$ for $\Gamma_{0}(4)$, in order to construct modular integrals of weight $2$ with rational period functions. In \cite{dit}, Duke, Imamoglu and T\'oth constructed a basis $\{h_{d}\}$ (indexed by discriminants $d > 0$) of $H_{1/2}$, which under $\xi_{1/2}$ maps to a basis $\{g_{d}\}$ of the space of weakly holomorphic modular forms of weight $3/2$ for $\Gamma_{0}(4)$. More precisely, the $g_{d}$ are the generating series of traces of singular moduli, see \cite{zagiertraces}. The coefficients of the $h_{d}$ are given by traces of CM values and traces of (regularized) cycle integrals of weakly holomorphic modular functions for $\Gamma$. For example, the Fourier expansion of the function $h = h_{1}$ is given by
\begin{align*}
h(\tau) = \frac{1}{2\pi}\sum_{D > 0}\tr_{J}(D)q^{D} + 2\sqrt{v}\beta_{1/2}^{c}(-4\pi v)q - 8\sqrt{v} + \sqrt{v}\sum_{D < 0}\tr_{J}(D)\beta_{1/2}(4\pi |D|v )q^{D},
\end{align*}
where
\[
\tr_{J}(D) = \begin{dcases}
 \sum_{Q \in \mathcal{Q}_{D}/\Gamma}\frac{J(z_{Q})}{|\overline{\Gamma}_{Q}|} , & D < 0, \\
\sum_{Q \in \mathcal{Q}_{D}/\Gamma}\int_{\Gamma_{Q}\setminus c_{Q}}J(z)\frac{dz}{Q(z,1)},  & D > 0,
\end{dcases}
\]
are traces of CM values and geodesic cycle integrals of $J = j-744$, which need to be regularized as explained in \cite{bif} if $D > 0$ is a square. The harmonic Maass form $h$ does not map to a cusp form but to a weakly holomorphic modular form under $\xi_{1/2}$, so it is interesting to apply our extension of the Borcherds lift to it. The coefficients $c_{h}^{+}(D)$ for $D \leq 0$ vanish, so the Borcherds lift $\Phi_{\Delta}(h,z)$ is a harmonic $\Gamma$-invariant function on $\H \setminus H_{\Delta}^{-}(h)$ with arcsin singularities along the geodesics in $H_{\Delta}^{-}(h)$. In this case, the latter set is just the union of all geodesics $c_{Q}$ for $Q \in \mathcal{Q}_{\Delta}$. Hence the derivative $\Phi'_{\Delta}(h,z) = \frac{\partial}{\partial z}\Phi_{\Delta}(h,z)$ is a holomorphic function on $\H \setminus H_{\Delta}^{-}(h)$ transforming like a modular form of weight $2$ for $\Gamma$. Moreover, it turns out that $\Phi'_{\Delta}(h,z)$ has jump singularities along the geodesics in $H_{\Delta}^{-}(h)$, and admits a nice Fourier expansion.

\begin{proposition}\label{proposition Borcherds lift derivative introduction}
	Let $\Delta > 1$ be a fundamental discriminant. The derivative $\Phi_{\Delta}'(h,z)$ of the Borcherds lift of $h$ is a holomorphic function on $\H \setminus H_{\Delta}^{-}(h)$ which transforms like a modular form of weight $2$ for $\Gamma$. For $z \in \H\setminus H_{\Delta}^{-}(h)$ it has the expansion
	\begin{align*}
	&\frac{1}{4\pi i \sqrt{\Delta }}\Phi'_{\Delta}(h,z) \\
	&= \frac{1}{2\pi}\tr_{1}(\Delta)+ \frac{1}{2\pi}\sum_{n=1}^{\infty}\left(\sum_{m \mid n} \left(\frac{\Delta}{n/m}\right)m\tr_{J}(\Delta m^{2}) \right) e(nz)  + \frac{1}{\pi}\sum_{\substack{Q \in \calQ_{\Delta}\\ a > 0}}\frac{\mathbf{1}_{Q}(z)}{Q(z,1)},
	\end{align*}
	where $\mathbf{1}_{Q}$ denotes the characteristic function of the bounded component of $\H \setminus c_{Q}$.
\end{proposition}

The result for general harmonic Maass forms $f \in H_{1/2}$ of higher level is given in Proposition \ref{proposition Fourier expansion derivative} and Corollary \ref{corollary derivative additional part}. 

\begin{remark} The Fourier series over $n$ is holomorphic on $\H$, whereas the sum over $Q$ has jump singularities along the geodesics $c_{Q}$ with $Q \in \calQ_{\Delta}$. Again, the sum over $Q$ is locally finite and vanishes for $y \gg 0$ large enough.
\end{remark}

\subsection{Modular integrals} In \cite{dit}, Theorem 5, the authors proved that the generating series
\[
F_{\Delta}(z) = \frac{1}{\pi}\sum_{m = 0}^{\infty}\tr_{J_{m}}(\Delta)e(mz),
\]
with $J_{m}(z) = q^{-m} + O(q) \in M_{0}^{!}$, e.g. $J_{0} = 1$ and $J_{1} = J$, defines a holomorphic function on $\H$ which transforms as
\begin{align}\label{eq transformation GDelta introduction}
z^{-2}F_{\Delta}\left(-\frac{1}{z}\right)-F_{\Delta}(z)  = \frac{2}{\pi}\sum_{\substack{Q \in \calQ_{\Delta} \\ c < 0 < a}}\frac{1}{Q(z,1)},
\end{align}
so $F_{\Delta}(z)$ is a holomorphic modular integral of weight $2$ with holomorphic rational period functions in the sense of \cite{knoppeichler}.

	Returning to the derivative of the Borcherds lift of $h$, we note that
\begin{align}\label{eq trace equation}
\tr_{J_{m}}(\Delta) = \sum_{d \mid m}\left( \frac{\Delta}{m/d}\right)d \tr_{J_{1}}(\Delta d^{2}),
\end{align}
compare \cite{zagiereisensteinriemann}, pp. 290--292, so $F_{\Delta}(z)$ in fact agrees with $\Phi'_{\Delta}(h,z)$ up to some constant factor if $y \gg 0$ is sufficiently large. The transformation behaviour of the singular part in the Fourier expansion of $\Phi'_{\Delta}(h,z)$ can easily be determined, so we can recover \eqref{eq transformation GDelta introduction} from Proposition \ref{proposition Borcherds lift derivative introduction}. Further, using the Borcherds lift we generalize the construction of modular integrals of weight $2$ with rational period functions from \cite{dit} to higher level, see Proposition \ref{proposition weight 2 modular integral}. The coefficients of our modular integrals are linear combinations of Fourier coefficients of the holomorphic parts of harmonic Maass forms $f$ of weight $1/2$. Choosing $f$ as the image of a theta lift of a harmonic Maass form $F$ of weight $0$ studied by Bruinier, Funke and Imamoglu \cite{bif}, we obtain modular integrals whose coefficients are linear combinations of traces of cycle integrals of $F$, see Example \ref{example weight 2 modular integral}. In fact, the construction of $F_{\Delta}$ as a theta lift and its generalizations to higher level were our main motivation to extend the Borcherds lift to the full space $H_{1/2}$.

\subsection{Borcherds products} Bruinier and Ono \cite{bruinieronoheegnerdivisors} defined a twisted Borcherds product associated to a harmonic Maass form $f \in H_{1/2}^{+}$ with real coefficients $c_{f}^{+}(D)$ for all $D$, and $c_{f}^{+}(D) \in \Z$ for $D \leq 0$. For $\Delta > 1$ a fundamental discriminant and $y \gg 0$ sufficiently large the twisted Borcherds lift of $f$ is given by
\[
\Psi_{\Delta}(f,z) = \prod_{m=1}^{\infty}\prod_{b (\Delta)}[1-e(mz+b/\Delta)]^{\left(\frac{\Delta}{b}\right)c_{f}^{+}(\Delta m^{2})}.
\]
It has a meromorphic continuation to $\H$ with roots and poles at CM points corresponding to the principal part of $f$, and it transforms like a modular form of weight $0$ with some unitary character for $\Gamma$. We will define Borcherds products associated to general harmonic Maass forms $f \in H_{1/2}$. For simplicity, in the introduction we only consider the harmonic Maass form $\pi h$. The general result is given in Theorem~\ref{theorem new borcherds products}.

\begin{theorem}
	Let $\Delta > 1$ be a fundamental discriminant. Then the infinite product 
	\begin{align*}
	\Psi_{\Delta}(z) &=  e\left( -\sqrt{\Delta}\tr_{1}(\Delta)z\right)\prod_{m=1}^{\infty}\prod_{b (\Delta)}[1-e(mz+b/\Delta)]^{\left(\frac{\Delta}{b}\right)\tr_{J}(\Delta m^{2})}
	\end{align*}
	converges to a holomorphic function on $\H$. Its logarithmic derivative is given by
	\begin{align*}
	\frac{\partial}{\partial z}\log(\Psi_{\Delta}(z)) = -2\pi^{2} i\sqrt{\Delta}F_{\Delta}(z).
	\end{align*} 
	Further, it transforms as
	\begin{align*}
	\Psi_{\Delta}(z+1) &= e\left( -\sqrt{\Delta}\tr_{1}(\Delta)\right)\Psi_{\Delta}(z), \\
	\Psi_{\Delta}\left(-\frac{1}{z}\right)&= e\left(-2\sum_{\substack{Q \in \calQ_{\Delta} \\ c < 0 < a}}\left(\log\left(\frac{z-w_{Q}}{i-w_{Q}}\right)-\log\left(\frac{z-w_{Q}'}{i-w_{Q}'}\right)\right) \right)\Psi_{\Delta}(z),
	\end{align*}
	where $w_{Q} > w_{Q}'$ denote the real endpoints of the geodesic $c_{Q}$.
	\end{theorem}
	
The work is organized as follows. We start with a section on the necessary preliminaries about the Grassmannian model of the upper half-plane, which is convenient for the study of regularized theta lifts, and vector valued harmonic Maass forms for the Weil representation. 

In Section \ref{section Borcherds lift}, we define the Borcherds lift of a harmonic Maass form of weight $1/2$ and prove its basic analytic properties.

In Section \ref{section Fourier expansion} we compute the Fourier expansion of the Borcherds lift. After that, in Section \ref{section Borcherds lift derivative} we compute the Fourier expansion of the derivative of the Borcherds lift and show that it has jump singularities along geodesics. 

In Section \ref{section modular integrals} we apply the derivative of the Borcherds lift to an interesting class of harmonic Maass forms which arise as images of regularized theta lifts of scalar valued weight $0$ harmonic Maass forms $F$. We show that the generating series of certain sums of traces of cycle integrals of $F$ transform like modular forms of weight $2$ for $\Gamma_{0}(N)$ and are holomorphic up to jump singularities along geodesics. Their non-singular parts yield holomorphic modular integrals of weight $2$ for $\Gamma_{0}(N)$ with rational period functions.

Finally, in Section \ref{section Borcherds products} we define the Borcherds product associated to a general harmonic Maass form of weight $1/2$ and prove its modularity.

\section{Preliminaries}
\label{section preliminaries}

\subsection{The Grassmannian model of the upper half-plane}
\label{section lattice}

For a positive integer $N$ we consider the quadratic space $V$ of all rational traceless $2$ by $2$ matrices, equipped with the quadratic form $Q(X) = N\det(X)$ and the associated bilinear form $(X,Y) = -N\tr(XY)$. It has signature $(1,2)$. We let $D$ be the Grassmannian of positive definite lines in $V(\R) = V\otimes \R$. We identify it with the complex upper half-plane $\H$ by associating to $z = x + iy \in \H$ the line spanned by
\begin{align*}
X(z) = \frac{1}{\sqrt{2N}y}\begin{pmatrix}-x & |z|^{2} \\ -1 & x \end{pmatrix}.
\end{align*}
The group $\SL_{2}(\R)$ acts as isometries on $V(\R)$ by conjugation and this action is compatible with the action by fractional linear transformations on $\H$ under the above identification. 

For $X \in V$ and $z \in D$ we let $X_{z}$ and $X_{z^{\perp}}$ denote the projection of $X$ to $z$ and its orthogonal complement $z^{\perp}$, respectively.

\subsection{A lattice related to $\Gamma_{0}(N)$}
In $V$ we consider the even lattice
\begin{align*}
	L = \left\{\begin{pmatrix}-b & -c/N \\ a & b\end{pmatrix}: a,b,c \in \Z \right\}.
\end{align*}
Its dual lattice is given by
\begin{align*}
	L' = \left\{\begin{pmatrix}-b/2N & -c/N \\ a & b/2N\end{pmatrix}: a,b,c \in \Z \right\}.
\end{align*}
We see that $L'/L \cong \Z/2N\Z$, and we will use this identification without further notice in the following. For $m \in \Q$ and $h \in L'/L$ we let
\begin{align*}
L_{m,h} = \{X \in L+h: Q(X) = m\}.
\end{align*}
The group $\Gamma_{0}(N)$ acts on $L_{m,h}$, with finitely many orbits if $m \neq 0$.

Let $X \in L_{m,h}$. If $m > 0$ we let 
\[
z_{X} = \spann(X) \in D
\]
 be the associated CM (or Heegner) point. If $m < 0$ we let 
 \[
 c_{X} = \{z \in D: z \perp X\}
 \]
 be the associated geodesic in $D$. We use the same symbols for the corresponding points and geodesics in $\H$. We can identify $X = [a,b,c] \in L'$ with the binary quadratic form $Q_{X} = [aN,b,c]$. Under this identification, the set $L_{m,h}$ corresponds to the set of all binary quadratic forms $[aN,b,c]$ with $a,b,c \in \Z$ of discriminant $-4Nm$ and $b \equiv h (2N)$. Furthermore, CM points and geodesics associated to vectors $X \in L_{m,h}$ correspond to the usual CM points and geodesics associated to $Q_{X}$. We define the quantity
 \[
 p_{X}(z) = -\frac{aN|z|^{2} +bx + c}{y\sqrt{N}},
\]
which vanishes exactly at the geodesic $c_{X}$.

\subsection{Vector valued harmonic Maass forms for the Weil representation}
\label{section harmonic maass forms}

Let $\tilde{\Gamma} = \Mp_{2}(\Z)$ be the integral metaplectic group, realized as the set of pairs $(M,\phi)$ with $M = \left(\begin{smallmatrix}a & b \\ c & d \end{smallmatrix}\right) \in \SL_{2}(\Z)$ and $\phi: \H \to \C$ holomorphic with $\phi^{2}(\tau) = c\tau + d$. We let $\tilde{\Gamma}_{\infty}$ be the subgroup of $\tilde{\Gamma}$ generated by $\tilde{T} =\left(\left(\begin{smallmatrix}1 & 1 \\ 0 & 1\end{smallmatrix}\right),1\right)$. Let $\C[L'/L]$ be the group ring of $L'/L$, generated by the formal basis vector $\e_{\gamma}$ for $\gamma \in L'/L$. We let $\langle \cdot,\cdot \rangle$ be the inner product on $\C[L'/L]$ which is antilinear in the second variable and satisfies $\langle \e_{\gamma},\e_{\beta}\rangle = \delta_{\gamma,\beta}$. The Weil representation $\rho_{L}$ is a unitary representation of $\tilde{\Gamma}$ on $\C[L'/L]$, see \cite{borcherds}, Section 4. We let $\rho_{L}^{*}$ be the dual Weil representation.

A harmonic Maass form of weight $1/2$ for $\rho_{L}$ is a harmonic function $f: \H \to \C[L'/L]$ which transforms like a modular form of weight $1/2$ for $\rho_{L}$ and is at most of linear exponential growth at $i\infty$. We let $H_{1/2,\rho_{L}}$ be the space of all harmonic Maass forms of weight $1/2$ for $\rho_{L}$. Every $f \in H_{1/2,\rho_{L}}$ can be written as a sum $f = f^{+} + f^{-}$ of a holomorphic and a non-holomorphic part having Fourier expansions of the form
\begin{align*}
f^{+} &= \sum_{h \in L'/L}\sum_{n \gg -\infty}c_{f}^{+}(n,h)e(n\tau)\e_{h}, \\
f^{-} &= \sum_{h \in L'/L}\bigg(c_{f}^{-}(0,h)\sqrt{v} + \sum_{n <0}c_{f}^{-}(n,h)\beta_{1/2}(-4\pi nv)e(n\tau) \\
&\qquad \qquad \qquad + \sum_{0< n \ll \infty}c_{f}^{-}(n,h)\beta_{1/2}^{c}(-4\pi nv)e(n\tau) \bigg)\e_{h},
\end{align*}
with coefficients $c_{f}^{\pm}(D) \in \C$, $e(x) = e^{2\pi i x}$ for $x \in \C$, and
\[
\beta_{1/2}(s) = \int_{1}^{\infty}e^{-st}t^{-1/2}dt, \qquad \beta_{1/2}^{c}(s) = \int_{0}^{1}e^{-st}t^{-1/2}dt.
\]
\begin{remark} For $N = 1$ the space $H_{1/2,\rho_{L}}$ of vector valued harmonic Maass forms is isomorphic to the space $H_{1/2}$ of scalar valued harmonic Maass forms from the introduction. The isomorphism is given by the map $f_{0}(\tau)\e_{0}+f_{1}(\tau)\e_{1} \mapsto f_{0}(4\tau) + f_{1}(4\tau)$, compare \cite{eichlerzagier}, Theorem 5.1. This identification can be used to translate the results from the body of the paper to the scalar valued setup in the introduction.
\end{remark}

\section{Analytic properties of the Borcherds lift}
\label{section Borcherds lift}

In this section we extend Borcherds' regularized theta lift to general harmonic Maass forms of weight $1/2$.

Let $\Delta \in \Z$ be a fundamental discriminant (possibly $1$) and let $r \in \Z/2N\Z$ such that $\Delta \equiv r^{2} \mod 4N$. We set 
\[
\tilde{\rho}_{L} = \begin{cases}
\rho_{L}, & \text{if } \Delta > 0, \\
\rho^{*}_{L}, & \text{if }\Delta < 0,
\end{cases}, \qquad 
Q_{\Delta}(X) = \frac{1}{|\Delta|}Q(X), \qquad (X,Y)_{\Delta} = \frac{1}{|\Delta|}(X,Y).
\]
We consider the twisted Siegel theta function
\begin{align}\label{eq definition theta function}
\Theta_{\Delta,r}(\tau,z) = v\sum_{h \in L'/L}\sum_{\substack{X\in L + rh \\ Q(X) \equiv \Delta Q(h) \, (\Delta)}}\chi_{\Delta}(X)e\big(\tau Q_{\Delta}(X_{z}) + \bar{\tau}Q_{\Delta}(X_{z^{\perp}})\big) \e_{h},
\end{align}
which is $\Gamma_{0}(N)$-invariant in $z$ and transforms like a modular form of weight $-1/2$ for $\tilde{\rho}_{L}$, compare \cite{bruinieronoheegnerdivisors}, Theorem 4.1. Here $\chi_{\Delta}$ is the genus character defined as in \cite{bruinieronoheegnerdivisors}, Section 4.

For $f \in H_{1/2,\tilde{\rho}^{*}_{L}}$ we let
\begin{align*}
H_{\Delta,r}^{+}(f) = \bigcup_{\substack{h \in L'/L,n < 0 \\ c_{f}^{+}(n,h) \neq 0} }\left\{z_{X}: X \in L_{-|\Delta|n,rh}\right\}, \qquad
H_{\Delta,r}^{-}(f) = \bigcup_{\substack{h \in L'/L,n > 0 \\ c_{f}^{-}(n,h) \neq 0} }\bigcup_{X \in L_{-|\Delta|n,rh}}c_{X},
\end{align*}
be the sets of Heegner points and geodesics associated to $f$.

Following Borcherds \cite{borcherds}, we define the regularized theta lift of $f \in H_{1/2,\tilde{\rho}^{*}_{L}}$ by
\begin{align}\label{eq definition theta lift}
\Phi_{\Delta,r}(f,z)  = \CT_{s = 0}\left(\lim_{T \to \infty}\int_{\mathcal{F}_{T}}\left\langle f(\tau),\overline{\Theta_{\Delta,r}(\tau,z)}\right\rangle v^{-s}\frac{du \, dv}{v^{2}}\right),
\end{align}
where 
\[
\mathcal{F}_{T} = \{\tau = u+iv \in \H:\, |\tau| \geq 1, \, |u| \leq 1/2, \, v  \leq T\}
\]
is a truncated fundamental domain for the action of $\SL_{2}(\Z)$ on $\H$, and $\CT_{s = 0}F(s)$ denotes the constant term in the Laurent expansion of the analytic continuation of $F(s)$ at $s = 0$. 

We say that a complex-valued function $f$ defined on some subset of $\R^{n}$ has a singularity of type $g$ (written $f \approx g$) at a point $z_{0}$ if there is an open neighbourhood $U$ of $z_{0}$ such that $f$ and $g$ are defined on a dense subset of $U$ and $f-g$ can be continued to a real analytic function on $U$.

\begin{theorem}\label{theorem Borcherds lift}
		For $f \in H_{1/2,\tilde{\rho}^{*}_{L}}$ the Borcherds lift $\Phi_{\Delta,r}(f,z)$ defines a $\Gamma_{0}(N)$-invariant real analytic function on $\H \setminus (H^{+}_{\Delta,r}(f) \cup H^{-}_{\Delta,r}(f))$ with
	 \[
	 \Delta_{0}\Phi_{\Delta,r}(f,z) = \begin{cases} -2c_{f}^{+}(0,0), & \text{if } \Delta = 1, \\
	 0, & \text{if } \Delta \neq 1.
	 \end{cases}
	 \]
	 At a point $z_{0} \in H^{+}_{\Delta,r}(f) \cup H^{-}_{\Delta,r}(f)$ it has a singularity of type
		\begin{align*}
		&-\sum_{h \in L'/L}\sum_{n < 0}c_{f}^{+}(n,h)\sum_{\substack{X\in L_{-|\Delta|n,rh} \\ z_{0} = z_{X}}}\chi_{\Delta}(X)\log(-Q_{\Delta}(X_{z^{\perp}}))\\
		&+
		\sum_{h \in L'/L}\sum_{n > 0}c_{f}^{-}(n,h)n^{-1/2}\sum_{\substack{X  \in L_{-|\Delta|n,rh} \\ z_{0} \in c_{X} }}\chi_{\Delta}(X)\arcsin\left(\sqrt{\frac{Q_{\Delta}(X)}{Q_{\Delta}(X_{z^{\perp}})}}\right).
		\end{align*}
	\end{theorem}

	\begin{remark}~
		\begin{enumerate}
			\item For $X =  \begin{pmatrix}-b/2N & -c/N \\ a & b/2N\end{pmatrix} \in L'$ we have 
			\[
			-Q_{\Delta}(X_{z^{\perp}}) = \frac{1}{4N|\Delta|y^{2}}|Q_{X}(z)|^{2} = \frac{1}{4N|\Delta|y^{2}}|aNz^{2} + bz + c|^{2},
			\]
			which yields a more explicit formula for the singularities. Since $	0 < \frac{Q_{\Delta}(X)}{Q_{\Delta}(X_{z^{\perp}})} \leq 1$ for all $X$ with $Q_{\Delta}(X) < 0$, and $Q_{\Delta}(X)/Q_{\Delta}(X_{z^{\perp}}) = 1$ exactly for $z \in c_{X}$, we see that the Borcherds lift extends to a continuous function on $\H \setminus H_{\Delta,r}^{+}(f)$, which is not differentiable along the geodesics in $H_{\Delta,r}^{-}(f)$. 
			Note that we can also write the singularities in the form
			\[
			\arcsin\left(\sqrt{\frac{Q_{\Delta}(X)}{Q_{\Delta}(X_{z^{\perp}})}}\right) = \arctan\left(\sqrt{\frac{Q_{\Delta}(X)}{-Q_{\Delta}(X_{z})}}\right).
			\]
			\item For every exact divisor $d \mid \mid N$ (i.e., $d \mid N$ and $(d,N/d) = 1$) the Atkin-Lehner involution $W_{d}$ acts on the Siegel theta function by
			\[
			\Theta_{\Delta,r}(\tau,W_{d}z)  = \Theta_{\Delta,r}(\tau,z)^{w_{d}},
			\]
			where $w_{d}$ is the orthogonal map on $L'/L$ defined by $w_{d}(h) \equiv -h (2d)$ and $w_{d}(h) \equiv h (2N/d)$, which acts on $f = \sum_{h}f_{h}\e_{h}$ by $f^{w_{d}} = \sum_{h}f_{h}\e_{w_{d}(h)}$. This implies that the Borcherds lift satisfies
			\[
			\Phi_{\Delta,r}(f,W_{d}z) = \Phi_{\Delta,r}(f^{w_{d}},z).
			\]		
			\end{enumerate}
	\end{remark}
	
	\begin{proof}[Proof of Theorem \ref{theorem Borcherds lift}]
		We first show that for $z \in \H \setminus (H_{\Delta,r}^{+}(f)\cup H_{\Delta,r}^{-}(f))$ the integral in \eqref{eq definition theta lift} converges absolutely and locally uniformly for $\Re(s) > 1/2$ and has a meromorphic continuation to $s = 0$. The proof follows the arguments of \cite{bruinierhabil}, Proposition~2.8. 
		
		The integral over the compact set $\mathcal{F}_{1}= \{\tau \in \H:\, |\tau| \geq 1, \, |u| \leq 1/2, \, v  \leq 1\}$ converges absolutely and locally uniformly for all $s \in \C$ and $z \in \H$. We consider the remaining integral
		\begin{align*}
		\varphi(z,s) &= \int_{v = 1}^{\infty}\int_{u = 0}^{1}\left\langle f(\tau),\overline{\Theta_{\Delta,r}(\tau,z)}\right\rangle v^{-s} \frac{du \, dv}{v^{2}}.
		\end{align*}
		Inserting the Fourier expansions of $f(\tau)$ and $\Theta_{\Delta,r}(\tau,z)$ and carrying out the integral over $u$, we obtain
		\begin{align*}
		&\varphi(z,s) = \chi_{\Delta}(0)\, \left( c_{f}^{+}(0,0)\int_{v=1}^{\infty}v^{-1-s}dv + c_{f}^{-}(0,0)\int_{v=1}^{\infty}v^{-1/2-s}dv\right) \\
		&\quad+ \int_{v = 1}^{\infty}\sum_{h,X}\chi_{\Delta}(X)c_{f}^{+}(-Q_{\Delta}(X),h)\exp\left(4\pi  Q_{\Delta}(X_{z^{\perp}})v\right)v^{-s-1}dv \\
		&\quad + \int_{v = 1}^{\infty}\sum_{Q(X) = 0}\chi_{\Delta}(X)c_{f}^{-}(-Q_{\Delta}(X),h)\exp(4\pi  Q_{\Delta}(X_{z^{\perp}})v)v^{-s-1/2}dv\\
		&\quad + \int_{v = 1}^{\infty}\sum_{Q(X) > 0}\chi_{\Delta}(X)c_{f}^{-}(-Q_{\Delta}(X),h)\beta_{1/2}(4\pi Q_{\Delta}(X)v)\exp(4\pi  Q_{\Delta}(X_{z^{\perp}})v)v^{-s-1/2}dv \\
		&\quad + \int_{v = 1}^{\infty}\sum_{Q(X) < 0}\chi_{\Delta}(X)c_{f}^{-}(-Q_{\Delta}(X),h)\beta_{1/2}^{c}(4\pi Q_{\Delta}(X)v)\exp(4\pi  Q_{\Delta}(X_{z^{\perp}})v)v^{-s-1/2}dv
		\end{align*}
		where the sums run over $h \in L'/L$ and $X \in (L+rh)\setminus\{0\}$ with $Q(X) \equiv \Delta Q(h) \mod \Delta$.
		
		Since $\chi_{\Delta}(0) =0$ for $\Delta \neq 1$ the integrals in the first line only appear if $\Delta = 1$. They can be evaluated for $\Re(s) > 1/2$ by
		\[
		\int_{v=1}^{\infty}v^{-1-s}dv = \frac{1}{s}, \qquad \int_{v=1}^{\infty}v^{-1/2-s}dv = \frac{1}{s-1/2},
		\]
		giving their meromorphic continuations to $s = 0$. Note that this shows that for $\Delta = 1$ the regularization involving the extra parameter $s$ is really necessary.
		
		The integral in the second line involving the coefficients $c_{f}^{+}(n,h)$ converges locally uniformly and absolutely for $s \in \C$ and $z \in \H \setminus H_{\Delta,r}^{+}(f)$ by the same arguments as in the proof of \cite{bruinierhabil}, Proposition~2.8. The integrals over the sums corresponding to $Q(X) = 0$ and $Q(X) > 0$ in the third and fourth line can be treated in the same way, and they converge locally uniformly and absolutely for $s \in \C$ and $z \in \H$.
				
		The remaining integral in the fifth line can be written as
		\begin{align*}
		&\sum_{h \in L'/L}\sum_{n > 0}c_{f}^{-}(n,h)\int_{v = 1}^{\infty}\sum_{X\in L_{-|\Delta|n,rh} }\chi_{\Delta}(X)\beta_{1/2}^{c}(4\pi Q_{\Delta}(X)v)\exp\left(4\pi Q_{\Delta}(X_{z^{\perp}})v\right)v^{-s-1/2}dv, \notag
		\end{align*} 
		where the first two sums are finite. Hence, estimating 
		\[
		\beta_{1/2}^{c}(4\pi Q_{\Delta}(X) v) \leq 2 \exp(-4\pi Q_{\Delta}(X) v)
		\]
		and using $Q_{\Delta}(X) = Q_{\Delta}(X_{z}) + Q_{\Delta}(X_{z^{\perp}})$, it suffices to consider the integral
		\begin{align}\label{eq integral nonholomorphic principal part}
		\int_{v = 1}^{\infty}\sum_{X \in L_{-|\Delta|n,rh} }\exp\left(-4\pi Q_{\Delta}(X_{z})v\right)v^{-\Re(s)-1/2}dv.
		\end{align}
		For any $C \geq 0$ and any compact subset $K \subset \H$ the set
		\[
		\left\{X \in L_{-|\Delta|n,rh}: \exists z \in K \text{ with } |Q_{\Delta}(X_{z})| \leq C\right\}
		\]
		is finite, so if $z \in K \subset \H \setminus H_{\Delta,r}^{-}(f)$ then there is some $\varepsilon > 0$ such that $Q_{\Delta}(X_{z}) > \varepsilon$ for all $X \in L_{-|\Delta|n,rh}$. We can now estimate
		\[
		\sum_{X \in L_{-|\Delta|n,rh} }\exp\left(-4\pi Q_{\Delta}(X_{z})v\right)\leq e^{-2\pi \varepsilon v}e^{\pi n}\sum_{X \in L_{-|\Delta|n,rh} }\exp\left(-\pi (Q_{\Delta}(X_{z}) - Q_{\Delta}(X_{z^{\perp}}))\right)
		\]
		for $v \geq 1$. The series on the right-hand side converges since $X \mapsto Q_{\Delta}(X_{z}) - Q_{\Delta}(X_{z^{\perp}})$ is a positive definite quadratic form. In particular, the integral in \eqref{eq integral nonholomorphic principal part} converges absolutely and locally uniformly for $s \in \C$ and $z \in \H \setminus H_{\Delta,r}^{-}(f)$. This shows that the regularized theta integral exists. 
		
		By similar arguments as above we see that all iterated partial derivatives of $\Phi_{\Delta,r}(f,z)$ converge absolutely and locally uniformly on $\H \setminus (H^{+}_{\Delta,r}(f) \cup H^{-}_{\Delta,r}(f))$, so the Borcherds lift is a smooth function. The statement concerning the Laplacian can now be proven by interchanging $\Delta_{0} = \Delta_{0,z}$ with the integral, using the differential equation
		\[
		\Delta_{0,z}\Theta_{\Delta,r}(\tau,z) = 4v^{1/2}\overline{\Delta_{1/2,\tau}v^{-1/2}\overline{\Theta_{\Delta,r}(\tau,z)}},
		\]
		(which can be checked by a direct calculation) and then applying Stokes' theorem to move $\Delta_{1/2,\tau}$ from the theta function to $f(\tau)v^{-s}$ in the integral (compare \cite{bruinierhabil}, Lemma 4.3). It is easy to verify that the appearing boundary integrals vanish. By computing $\Delta_{1/2}(f(\tau)v^{-s})$ explicitly and using that $f$ is harmonic, we obtain 
		\[
		\Delta_{0}\Phi_{\Delta,r}(f,z) = -2\Res_{s = 0}\lim_{T \to \infty}\int_{\mathcal{F}_{T}}\langle f(\tau),\overline{\Theta_{\Delta,r}(\tau,z)}\rangle v^{-s}\frac{du \, dv}{v^{2}}.
		\]
		We have seen above that the integral on the right-hand side is holomorphic at $s = 0$ if $\Delta \neq 1$, and has a simple pole with residue $a_{f}^{+}(0,0)$ if $\Delta = 1$, coming from the first integral in the first line of $\varphi(z,s)$. This shows the Laplace equation for $\Phi_{\Delta,r}(f,z)$, which also implies that the Borcherds lift is real analytic by a standard regularity result for elliptic differential equations.
		
		The singularities of $\Phi_{\Delta,r}(f,z)$ can be determined using the following lemma with $n = -Q_{\Delta}(X)$ and $t = -Q_{\Delta}(X_{z^{\perp}})$.
	\end{proof}
	
	\begin{lemma}
		\begin{enumerate}
			\item The function
			\[
			I^{+}(t) = \int_{v=1}^{\infty}e^{-4\pi tv }\frac{dv}{v}
			\]
			is real analytic for $t > 0$ and has a singularity of type $-\log(t)$ at $t = 0$.
			\item For $n > 0$ the function
			\[
			I^{-}_{n}(t) = \int_{v=1}^{\infty}\sqrt{v}\beta_{1/2}^{c}(-4\pi nv)e^{-4\pi tv}\frac{dv}{v}
			\]
			is real analytic for $t > n$ and has a singularity of type $n^{-1/2}\arcsin\left(\sqrt{\frac{n}{t}}\right)$ at $t = n$.
		\end{enumerate}
	\end{lemma}
	
	\begin{proof}
		We follow the proof of \cite{borcherds}, Lemma 6.1. Using partial integration and the fact that $\log(v)$ is integrable near $v = 0$, we see that 
		\[
		I^{+}(t) \approx  4\pi t \int_{v = 0}^{\infty}e^{-4\pi tv }\log(v)dv = \int_{v =0}^{\infty}e^{-v}\log\left(\frac{v}{4\pi t}\right)dv \approx -\log(t).
		\]
		For $n > 0$, we use that $\sqrt{v}\beta_{1/2}^{c}(-4\pi n v) = O(\sqrt{v})$ as $v \to 0$ and compute
		\begin{align*}
		I_{n}^{-}(t) &\approx \int_{v = 0}^{\infty}\left(2\sqrt{v}\int_{w = 0}^{1}e^{4\pi nv w^{2}}dw \right)e^{-4\pi t v }\frac{dv}{v} \\
		 &= \int_{w=0}^{1}\frac{1}{\sqrt{t-nw^{2}}}dw \\
		 &= n^{-1/2}\arcsin\left( \sqrt{\frac{n}{t}}\right).
		\end{align*}
		This finishes the proof of the lemma and of Theorem \ref{theorem Borcherds lift}.
	\end{proof}

\section{The Fourier expansion of the Borcherds lift} \label{section Fourier expansion}

Next, we compute the Fourier expansion of the Borcherds lift. To this end, we first need to introduce a special function which captures the arcsin singularities of $\Phi_{\Delta,r}(f,z)$ along vertical geodesics.

	For $a \geq 1$ and $\Re(s) > -1$ we define
	\begin{align}\label{eq arcsin}
	\arcsin_{s}\left(\frac{1}{\sqrt{a}}\right) = \int_{0}^{1}\frac{1}{\sqrt{a-t^{2}}} \left(\frac{1-t^{2}}{a - t^{2}} \right)^{s}dt.
	\end{align}
	The function $\arcsin_{s}$ is holomophic in $s$ and satisfies 
	\[
	\arcsin_{0}(1/\sqrt{a}) = \arcsin(1/\sqrt{a}).
	\]
	The factor $(1-t^{2})^{s}$ ensures that the integral converges at $a = 1$ if $\Re(s) \geq 1/2$, and the factor $(a-t^{2})^{s}$ in the denominator was added to make the estimate 
	\begin{align}\label{eq arcsin estimate}
	|\arcsin_{s}(1/\sqrt{a})| \leq (a-1)^{-\Re(s)-1/2}
	\end{align} for $a > 1$ and $\Re(s) > 0$ hold. Note that for $\Re(s) > -1$ we can write
	\begin{align}\label{eq arcsin by incomplete beta}
	\arcsin_{s}\left(\frac{1}{\sqrt{a}}\right) = \frac{\sqrt{\pi}\,\Gamma(s+1)}{2\Gamma(s+1/2)}B(1/a;s+1/2,1/2),
	\end{align}
	where 
	\[
	B(z;\alpha,\beta) = \int_{0}^{z}u^{\alpha-1}(1-u)^{\beta-1}du
	\]
	is the incomplete beta function.	
	\begin{lemma}\label{lemma Fourier expansion arcsin series}
	For $z = x + iy \in \mathbb{H}$ and $\Re(s) > 0$ we have the Fourier expansion
	\begin{align*}
	&\sum_{\ell \in \Z}\arcsin_{s}\left(\frac{y}{\sqrt{(x+\ell)^{2} + y^{2}} }\right)= y\frac{\sqrt{\pi}\,\Gamma(s)}{\Gamma(s+1/2)}  \\
	& \qquad+ 2y\frac{\sqrt{\pi}}{\Gamma(s+1/2)}\sum_{n \neq 0}(\pi |n|y)^{s}\left(\int_{0}^{1}(1-t^{2})^{s/2}K_{s}\left(2\pi |n| y\sqrt{1-t^{2}}\right)dt\right) \cos(2\pi n x),
	\end{align*}
	where $K_{s}$ denotes the $K$-Bessel function of order $s$. For $\Re(s) > -1$ the series on the right-hand side converges absolutely and locally uniformly in $s$. In particular, the left-hand side has a meromorphic continuation to $\Re(s) > -1$ with a simple pole at $s = 0$.
	\end{lemma}
	
	\begin{proof}
		The estimate \eqref{eq arcsin estimate} shows that the series on the left-hand side converges absolutely for $\Re(s) > 0$. It is $1$-periodic and even in $x$ and hence has a Fourier expansion of the form $\sum_{n \in \Z}a(n,y)\cos(2\pi nx)$ with coefficients
		\begin{align*}
		a(n,y) &= 
		\int_{-\infty}^{\infty}\arcsin_{s}\left(\frac{1}{\sqrt{(u/y)^{2} + 1} }\right)\cos(2\pi nu)du.
		\end{align*}
		We plug in the definition of $\arcsin_{s}$ and interchange the order of integration to find
		\begin{align*}
		a(n,y) &= \int_{0}^{1}\left(\int_{-\infty}^{\infty}\frac{\cos(2\pi n u)}{((u/y)^{2}+1-t^{2})^{s+1/2}} du \right) (1-t^{2})^{s} dt \\
		&=  y\int_{0}^{1}\left(\int_{-\infty}^{\infty}\frac{\cos\left(2\pi n uy\sqrt{1-t^{2}}\right)}{(u^{2}+1)^{s+1/2}} du \right) dt.
		\end{align*}
		For $n = 0$ the inner integral can be evaluated as
		\[
		\int_{-\infty}^{\infty}\frac{1}{(u^{2}+1)^{s+1/2}}du = \frac{\sqrt{\pi}\, \Gamma(s)}{\Gamma(s+1/2)},
		\]
		by a direct calculation using the definition of the Gamma function. For $n \neq 0$ we can replace $n$ by $|n|$, and then the inner integral can be computed using the representation
		\[
		K_{s}(x) = \frac{2^{s-1}\Gamma(s+1/2)}{\sqrt{\pi}\, x^{s}}\int_{-\infty}^{\infty}\frac{\cos(xu)}{(u^{2}+1)^{s+1/2}}du,
		\]
		which is valid for $\Re(s) > -1/2$ and $x > 0$ (see \cite[9.6.25]{abramowitz}). Note that $K_{s} = K_{-s}$.
		
		The asymptotics $K_{0}(z) \sim -\log(z)$ and $K_{s}(z) \sim \frac{1}{2}\Gamma(s)(\frac{1}{2}z)^{-s}$ for $\Re(s) > 0$ fixed as $z \to 0$ (see \cite[9.6.8, 9.6.9]{abramowitz}) show that the integral in the series is holomorphic for $\Re(s) > -1$. This completes the proof.
	\end{proof}

	\begin{proposition}	\label{proposition Fourier expansion}
		Let $f \in H_{1/2,\tilde{\rho}^{*}_{L}}$. For $y \gg 0$ sufficiently large, the Borcherds lift of $f$ has the Fourier expansion
		\begin{align*}
		\Phi_{\Delta,r}(f,z) &=
			-4\sum_{m=1}^{\infty}c_{f}^{+}(|\Delta|m^{2}/4N,rm)\sum_{b (\Delta)}\left( \frac{\Delta}{b}\right)\log|1-e(mz + b/\Delta)| \\
			&\quad+ 2\sum_{m = 1}^{\infty}c_{f}^{-}(|\Delta|m^{2}/4N,rm)\left( \frac{|\Delta|m^{2}}{4N}\right)^{-1/2}\sum_{b (\Delta)}\left(\frac{\Delta}{b}\right)\mathcal{F}(mz + b/\Delta) \\
			&\quad+ \begin{dcases} 
				\begin{aligned}
					&\sqrt{N}y(f,\theta_{1/2})^{\reg} - c_{f}^{+}(0,0)\left(\log(4\pi Ny^{2})+\Gamma'(1) \right) \\
					&\quad - \sqrt{N}y\, c_{f}^{-}(0,0)\left(\log(4\pi)-\log(Ny^{2})+\Gamma'(1)\right)
				\end{aligned}
				& \text{if }\Delta = 1, \\
			2\sqrt{\Delta}\, L_{\Delta}(1) \big(c_{f}^{+}(0,0) + \sqrt{N}y\, c_{f}^{-}(0,0)\big)  & \text{if }\Delta > 1, \\
			0 & \text{if } \Delta < 0,
			\end{dcases}
		\end{align*}
		where $\theta_{1/2}(\tau) = \sum_{h \in L'/L}\sum_{n \in h + 2N\Z}e(n^{2}\tau/4N)\e_{h}$ is the Jacobi theta function and $L_{\Delta}(s) = \sum_{n \geq 1}\left( \frac{\Delta}{n}\right)n^{-s}$ for $\Re(s)>1$ is a Dirichlet $L$-function. Here the function $\mathcal{F}(z): \H \to \R$ is defined by
		\begin{align*}
		\mathcal{F}(z) = \lim_{s \to 0}\left(\sum_{\ell \in \Z}\arcsin_{s}\left( \frac{y}{\sqrt{(x+\ell)^{2} + y^{2}}}\right)-y\frac{\sqrt{\pi}\,\Gamma(s)}{\Gamma(s+1/2)}\right),
		\end{align*}
		compare Lemma \ref{lemma Fourier expansion arcsin series}.
	\end{proposition} 
	
	\begin{remark}
		\begin{enumerate}
			\item The singularities of $\Phi_{\Delta,r}(f,z)$ at Heegner points and geodesics given by semi-circles centered at the real line are not reproduced in the Fourier expansion above, but the part involving the function $\mathcal{F}$ captures the singularities along vertical geodesics.
			\item By Dirichlet's class number formula we have
				\[
				L_{\Delta}(1) = \frac{1}{\sqrt{\Delta}}h(\Delta)\log(\epsilon_{\Delta}) = \frac{1}{2}\tr_{1}(\Delta)
				\]
				for $\Delta > 1$, where $h(\Delta)$ is the narrow class number of $\Q(\sqrt{\Delta})$, $\epsilon_{\Delta}$ is the smallest unit $>1$ of norm $1$, and $\tr_{1}(\Delta)$ is the $\Delta$-th trace of the constant $1$ function as defined in the introduction.
		\end{enumerate}
	\end{remark}
	
	\begin{proof}[Proof of Proposition \ref{proposition Fourier expansion}]
		The proof follows the arguments of \cite{bruinieronoheegnerdivisors}, Theorem 5.3. First, by \cite{bruinieronoheegnerdivisors}, Theorem 4.8, we can write
		\begin{align*}
		v^{-1/2}\overline{\Theta_{\Delta,r}(\tau,z)} &= \delta_{\Delta = 1}\frac{\sqrt{N}y}{\sqrt{|\Delta|}}\theta_{1/2}(\tau)  \\
		& \quad + \frac{\sqrt{N}y}{\sqrt{|\Delta|}}\sum_{n \geq 1}\sum_{M \in \tilde{\Gamma}_{\infty}\setminus \tilde{\Gamma}}\left[\exp\left( -\frac{\pi n^{2}Ny^{2}}{|\Delta|v}\right)\Xi(\tau,\mu,n,0)\right]\bigg|_{1/2,\tilde{\rho}_{L}^{*}}M,
		\end{align*}
		where $\mu = \big(\begin{smallmatrix}x & -x^{2} \\ -1 & -x \end{smallmatrix}\big)$ and
		\[
		\Xi(\tau,\mu,n,0) = \left(\frac{\Delta}{n}\right)\overline{\varepsilon}\sqrt{|\Delta|}\sum_{h \in K'/K}\sum_{\substack{X \in K+rh \\ Q(X) \equiv \Delta Q(h) \, (\Delta)}} e\left( -Q_{\Delta}(X)\tau + n(X,\mu)_{\Delta}\right)\e_{h},
		\]
		with $\varepsilon = 1$ if $\Delta > 0$ and $\varepsilon = i$ if $\Delta < 0$. Further, $K$ denotes the one-dimensional negative definite sublattice 
	\[
	K = \left\{\begin{pmatrix}b & 0 \\ 0 & -b \end{pmatrix}: b \in \Z\right\}
	\]
	of $L$. Its dual lattice is given by
	\[
	K' = \left\{\begin{pmatrix}b/2N & 0 \\ 0 & -b/2N \end{pmatrix}: b \in \Z\right\}.
	\]
  Inserting this into the definition of the theta lift, the unfolding argument yields
		\begin{align*}
		\Phi_{\Delta,r}(f,z) &= \delta_{\Delta = 1}\frac{\sqrt{N}y}{\sqrt{|\Delta|}}(f,\theta_{1/2})^{\reg} +\CT_{s = 0}\Phi^{0}_{\Delta,r}(f,z,s),
		\end{align*}
		where
		\[
		\Phi^{0}_{\Delta,r}(f,z,s) = \frac{2\sqrt{N}y}{\sqrt{|\Delta|}}\sum_{n\geq 1}\int_{v=0}^{\infty}\int_{u=0}^{1}\exp\left(-\frac{\pi n^{2}Ny^{2}}{|\Delta|v}\right)\langle f,\Xi(\tau,\mu,n,0)\rangle du \frac{dv}{v^{s+3/2}}.
		\]
		The unfolding is justified for $y \gg 0$ by the same arguments as in \cite{borcherds}, Theorem 7.1. Let us write
		\[
		f(\tau) = \sum_{h \in L'/L}\sum_{n \in \Q}c_{f}(n,h,v)e(n\tau)\e_{h}
		\]
		for the Fourier expansion of $f$ for the moment. Since $\Delta$ is fundamental, the conditions $X \in K + rh$ and $Q(X)\equiv \Delta Q(h) \mod \Delta$ are equivalent to $X =\Delta X'$ and $rX' \in K + h$ for some $X' \in K'$. Plugging in the definition of $\Xi(\tau,n,\mu,0)$, and evaluating the integral over $u$, we obtain
		\begin{align*}
			\Phi_{\Delta,r}^{0}(f,z,s) &= 2\sqrt{N}y\varepsilon\sum_{X \in K'}\sum_{n\geq 1}\left( \frac{\Delta}{n}\right)e\left(-\sgn(\Delta)n(X,\mu)\right) \\
			&\quad \times \int_{v = 0}^{\infty}c_{f}(-|\Delta|Q(X),rX,v)\exp\left( -\frac{\pi n^{2}Ny^{2}}{|\Delta|v} + 4\pi |\Delta|Q(X)v\right)\frac{dv}{v^{s+3/2}}.
		\end{align*}
		Now we use the explicit form of the Fourier coefficients of $f$. The summand for $X = 0$ in $\Phi^{0}_{\Delta,r}(f,z,s)$ is given by
		\begin{align*}
			&2\sqrt{N}y\varepsilon\sum_{n\geq 1}\left( \frac{\Delta}{n}\right)\int_{v = 0}^{\infty}\left(c_{f}^{+}(0,0) + c_{f}^{-}(0,0)v^{1/2} \right)\exp\left( -\frac{\pi n^{2}Ny^{2}}{|\Delta|v} \right) \frac{dv}{v^{s+3/2}} \\
			&= 2\varepsilon \left( Ny^{2}\right)^{-s}\bigg(c_{f}^{+}(0,0)\left(\frac{\pi}{|\Delta|}\right)^{-s-1/2}\Gamma(s+1/2)L_{\Delta}(2s+1) \\
			&\qquad \qquad \qquad \quad  + \sqrt{N}y \, c_{f}^{-}(0,0)\left(\frac{\pi}{|\Delta|}\right)^{-s}\Gamma(s)L_{\Delta}(2s)\bigg).
		\end{align*}
		For $\Delta < 0$ the harmonic Maass form $f$ transforms with $\rho_{L}$, which implies that its zero component vanishes, so $c_{f}^{\pm}(0,0) = 0$. For $\Delta > 0$ the completed Dirichlet $L$-function 
		\[
		\Lambda_{\Delta}(s) = (\pi/\Delta)^{-s/2}\Gamma(s/2)L_{\Delta}(s)
		\]
		satisfies the functional equation $\Lambda_{\Delta}(1-s) = \Lambda_{\Delta}(s)$. It is holomorphic at $s = 1$ if $\Delta > 1$. Taking the constant term at $s = 0$, we get the contribution in the large bracket in the proposition.

		For $X \in K'$ with $X \neq 0$ we have $-|\Delta|Q(X) > 0$. We can write
		\[
		c_{f}(n,h,v) = c_{f}^{+}(n,h) + c_{f}^{-}(n,h)\sqrt{v}\beta_{1/2}^{c}(-4\pi n v)
		\]
		for $n > 0$. The contribution coming from the coefficients $c_{f}^{+}(n,h)$ can be computed as in \cite{bruinieronoheegnerdivisors}, Theorem 5.2, and yields the first line of the Fourier expansion. Plugging in the definition of $\beta_{1/2}^{c}(s)$, it remains to compute
		\begin{align}\label{eq contribution positive nonholomorphic}
		\begin{split}
			&4\sqrt{N}y\varepsilon\sum_{\substack{X \in K' \\ X \neq 0}}c_{f}^{-}(-|\Delta| Q(X),rX)\sum_{n \geq 1}\left( \frac{\Delta}{n}\right)e\left(-\sgn(\Delta)n(X,\mu)\right) \\
			&\quad \times \int_{v = 0}^{\infty}\left(\int_{w =0}^{1}\exp(-4\pi |\Delta|Q(X) w^{2}v)dw\right)\exp\left( -\frac{\pi n^{2}Ny^{2}}{|\Delta|v} + 4\pi |\Delta|Q(X)v\right)\frac{dv}{v}. 
			\end{split}
		\end{align}
		If we change the order of integration, the inner integral can be computed in terms of the $K$-Bessel function by \cite[(3.471.9)]{tablesofintegrals}, giving
		\[
		\int_{v = 0}^{\infty}\exp\left(4\pi |\Delta|Q(X)(1-w^{2})v -\frac{\pi n^{2}Ny^{2}}{|\Delta|v}\right)\frac{dv}{v} = 2K_{0}\left(2\pi y |n|\sqrt{-4NQ(X)(1-w^{2})}\right).
		\]
		Write $X = \left(\begin{smallmatrix}m/2N & 0 \\ 0 & -m/2N \end{smallmatrix}\right) \in K' \setminus \{0\}$ with $m \in \Z, m \neq 0$. Then $-Q(X) = m^{2}/4N$ and $-(X,\mu) = mx$. We use the evaluation of the Gauss sum 
		\begin{align}\label{eq Gauss sum}
		\sum_{b(\Delta)}\left( \frac{\Delta}{b}\right)e(bn/|\Delta|) = \varepsilon \left(\frac{\Delta}{n}\right)\sqrt{|\Delta|}.
		\end{align}
		Then the expression in \eqref{eq contribution positive nonholomorphic} becomes
		\begin{align*}
		& 2\sum_{m = 1}^{\infty}a_{f}(|\Delta|m^{2}/4N,rm)\left( \frac{|\Delta|m^{2}}{4N}\right)^{-1/2}\sum_{b(\Delta)}\left(\frac{\Delta}{b}\right) \\
		&\quad \times 2my\sum_{n \neq 0}e\left(n(mx+b/\Delta)\right)\int_{0}^{1}K_{0}\left(2\pi my|n|\sqrt{1-w^{2}} \right)dw.
		\end{align*}
		By Lemma \ref{lemma Fourier expansion arcsin series} the second line agrees with $\mathcal{F}(mz + b/\Delta)$, which finishes the proof.
		\end{proof}
		
		If the coefficients $c_{f}^{+}(n,h)$ vanish for $n < 0$, then $\Phi_{\Delta,r}(f,z)$ does not have singularities at Heegner points, and extends to a continuous function on $\H$ which is not differentiable along the geodesics in $H_{\Delta,r}^{-}(f)$. In this case, we can derive the Fourier expansion of $\Phi_{\Delta,r}(f,z)$ on $\H$, without assuming $y \gg 0$ to be large enough.
	
		\begin{corollary}\label{corollary Borcherds lift Fourier expansion additional part}
			Let $f \in H_{1/2,\tilde{\rho}^{*}_{L}}$, and suppose that $c_{f}^{+}(n,h) = 0$ for all $n <0$ and $h \in L'/L$. Then the Fourier expansion of the Borcherds lift $\Phi_{\Delta,r}(f,z)$ on $\H$ is given by the formula from Proposition \ref{proposition Fourier expansion} plus the expression
			\begin{align}\label{eq Borcherds lift Fourier expansion additional part}
			-2\sum_{h \in L'/L}\sum_{n > 0}c_{f}^{-}(n,h)n^{-1/2}\sum_{\substack{X \in L_{-|\Delta|n,rh} \\ a \neq 0}}\chi_{\Delta}(X)\mathbf{1}_{X}(z)\left(\arctan\left( \frac{\sqrt{4|\Delta|n}}{-\sgn(a)p_{X}(z)}\right) + \frac{\pi}{2}\right),
			\end{align}
			where $\mathbf{1}_{X}(z)$ denotes the characteristic function of the bounded component of $\H \setminus c_{X}$.
		\end{corollary}
		
		\begin{remark}
			\begin{enumerate}
				\item Recall that for $X =  \begin{pmatrix}-b/2N & -c/N \\ a & b/2N\end{pmatrix} \in L'$ we defined 
			\[
			p_{X}(z) = -\frac{aN|z|^{2} +bx + c}{y\sqrt{N}},
			\]
			which vanishes exactly along the geodesic $c_{X}$. Further, if $a \neq 0$ then a point $z$ lies inside the bounded component of $\H \setminus c_{X}$ if and only if $\sgn(a)p_{X}(z) > 0$. In particular, we see that if $z \in \H \setminus c_{X}$ approaches $c_{X}$, then the expression
			\[
			\mathbf{1}_{X}(z)\left(\arctan\left( \frac{\sqrt{4|\Delta|n}}{-\sgn(a)p_{X}(z)}\right) + \frac{\pi}{2}\right)
			\] 
			goes to $0$. In this sense, the above Fourier expansion is defined on all of $\H$.
				\item The sum in \eqref{eq Borcherds lift Fourier expansion additional part} is locally finite since for fixed $n$ each point $z$ lies in the bounded component of $\H\setminus c_{X}$ for only finitely many $X \in L_{-|\Delta| n,rh}$ with $a \neq 0$.
			\end{enumerate}
		\end{remark}
		
		\begin{proof}
			Let $\tilde{\Phi}_{\Delta,r}(f,z)$ denote $\Phi_{\Delta,r}(f,z)$ minus the expression in \eqref{eq Borcherds lift Fourier expansion additional part}. Then we have $\tilde{\Phi}_{\Delta,r}(f,z) = \Phi_{\Delta,r}(f,z)$ for $y \gg 0$ large enough since the imaginary parts of points lying on geodesics $c_{X}$ for $X \in L_{-|\Delta|n,rh}$ with $a \neq 0$ are bounded by a constant depending on $n$, and the sum over $n$ is finite.
			
			Further, for $a \neq 0$ and $z \notin c_{X}$ we can write
			\begin{align*}
			&-2\cdot \mathbf{1}_{X}(z)\left(\arctan\left( \frac{\sqrt{4|\Delta|n}}{-\sgn(a)p_{X}(z)}\right) + \frac{\pi}{2}\right) \\
			&= \arctan\left( \frac{\sqrt{4|\Delta|n}}{|p_{X}(z)|}\right) - \left( \arctan\left( \frac{\sqrt{4|\Delta|n}}{-\sgn(a)p_{X}(z)}\right) + \mathbf{1}_{X}(z) \pi\right) \\
			&= \arctan\left( \sqrt{\frac{Q_{\Delta}(X)}{-Q_{\Delta}(X_{z})}}\right) - \arccot\left( \frac{-\sgn(a)p_{X}(z)}{\sqrt{4|\Delta|n}} \right).
			\end{align*}
			Using that the function $\arccot$ is real analytic at the origin, and the shape of the singularities of $\Phi_{\Delta,r}(f,z)$ determined in Theorem \ref{theorem Borcherds lift}, we see that $\tilde{\Phi}_{\Delta,r}(f,z)$ extends to a real analytic on all of $\H$. In particular, the Fourier expansion of $\Phi_{\Delta,r}(f,z)$ given in Proposition \ref{proposition Fourier expansion}, which a priori only converges for $y \gg 0$ sufficiently large, is also the Fourier expansion of the real analytic function $\tilde{\Phi}_{\Delta,r}(f,z)$ on all of $\H$, and hence converges on all of $\H$. We obtain the stated Fourier expansion.
		\end{proof}

		\section{The derivative of the Borcherds lift}\label{section Borcherds lift derivative}
		
		We consider the derivative 
		\[
		\Phi'_{\Delta,r}(f,z) = \frac{\partial}{\partial z}\Phi_{\Delta,r}(f,z)
		\]
		of the Borcherds lift. 
		
		\begin{theorem}\label{theorem Borcherds lift derivative}
			Let $f \in H_{1/2,\tilde{\rho}^{*}_{L}}$. The derivative $\Phi'_{\Delta,r}(f,z)$ of the Borcherds lift is harmonic on $\H \setminus (H_{\Delta,r}^{+}(f) \cup H_{\Delta,r}^{-}(f))$ and transforms like a modular form of weight $2$ under $\Gamma_{0}(N)$. If $\Delta \neq 1$ or if $c_{f}^{+}(0,0) = 0$, then $\Phi'_{\Delta,r}(f,z)$ is holomorphic on its domain. 
			
			At a point $z_{0} \in H^{+}_{\Delta,r}(f) \cup H^{-}_{\Delta,r}(f)$ it has a singularity of type
		\begin{align*}
		&i\sqrt{N}\sum_{h \in L'/L}\sum_{n < 0}c_{f}^{+}(n,h)\sum_{\substack{X \in L_{-|\Delta|n,rh} \\  z_{0} = z_{X}}}\chi_{\Delta}(X)\frac{p_{X}(z)}{Q_{X}(z)}\\
		&+i\sqrt{N|\Delta|}\sum_{h \in L'/L}\sum_{n> 0}c_{f}^{-}(n,h)\sum_{\substack{X \in L_{-|\Delta|n,rh} \\ z_{0} \in c_{X}}}\chi_{\Delta}(X)\frac{\sgn(p_{X}(z))}{Q_{X}(z)}.
		\end{align*}
		\end{theorem}
		
		\begin{proof}
			The analytic properties of $\Phi'_{\Delta,r}(f,z)$ follow from the Laplace equation in Theorem \ref{theorem Borcherds lift} and the formula $\Delta_{0} = -4y^{2}\frac{\partial}{\partial \bar{z}}\frac{\partial}{\partial z}$. The types of singularities of $\Phi'_{\Delta,r}(f,z)$ are obtained as the derivatives of the types of singularities of $\Phi_{\Delta,r}(f,z)$.
		\end{proof}
		
		\begin{remark} Let $X = \begin{pmatrix}-b/2N & -c/N \\ a & b/2N \end{pmatrix} \in L_{-|\Delta|n,rh}$. For $n < 0$ we have 
		\[
		Q_{X}(z) = aNz^{2} + bz + c = 0
		\]
		exactly for the Heegner point $z = z_{X}$. Hence $\Phi'_{\Delta,r}(f,z)$ has simple poles at the Heegner points in $H_{\Delta,r}^{+}(f)$. For $n > 0$ the sign of 
		\[
		p_{X}(z) = -\frac{aN|z|^{2} + bx + c}{y\sqrt{N}}
		\]
		 changes if $z$ crosses the geodesic $c_{X}$. This means that $\Phi'_{\Delta,r}(f,z)$ has jump singularities along the geodesics in $H_{\Delta,r}^{-}(f)$.
		\end{remark}

		\begin{proposition}\label{proposition Fourier expansion derivative}
			Let $f \in H_{1/2,\tilde{\rho}^{*}_{L}}$. For $y \gg 0$ sufficiently large we have the Fourier expansion
			\begin{align*}
			\Phi'_{\Delta,r}(f,z) &= 4 \pi i \sqrt{|\Delta|}\bar{\varepsilon} \sum_{n = 1}^{\infty}\left(\sum_{d \mid n}\left( \frac{\Delta}{n/d}\right)d \, c_{f}^{+}\left(|\Delta|d^{2}/4N,rd\right)\right)e(nz) \\
			&\quad + 2\sum_{m=1}^{\infty}c_{f}^{-}(|\Delta|m^{2}/4N,rm)\left(\frac{|\Delta|}{4N} \right)^{-1/2}\sum_{b(\Delta)}\left( \frac{\Delta}{b}\right)\mathcal{F}'(mz+b/\Delta) \\
			&\quad+ \begin{cases} 
				\begin{aligned}
					&-\frac{i\sqrt{N}}{2}(f,\theta_{1/2})^{\reg} + \frac{i}{y}c_{f}^{+}(0,0) \\
					&\quad + \frac{i}{2}\sqrt{N}\, c_{f}^{-}(0,0)\left(\log(4\pi)-\log(Ny^{2})-2+\Gamma'(1)\right)
				\end{aligned}
				& \text{if }\Delta = 1, \\
			-i \sqrt{N\Delta}\, L_{\Delta}(1) c_{f}^{-}(0,0)  & \text{if }\Delta > 1, \\
			0 & \text{if } \Delta < 0.
			\end{cases}
			\end{align*}
			where $\varepsilon = 1$ if $\Delta > 0$ if $\varepsilon = i$ for $\Delta < 0$, and
			\[
			\mathcal{F}'(z) = -\frac{i}{2}\lim_{s \to 0}\left( y^{2s}\Gamma(s+1)\sum_{\ell \in \Z}\frac{\sgn(x+\ell)(\bar{z}+\ell)}{|z+\ell|^{2s+2}} - \Gamma(s)\right).
			\]
		\end{proposition}
		
		\begin{proof}
			The derivative of $\mathcal{F}(z)$ can be computed most easily using the representation \eqref{eq arcsin by incomplete beta} of $\arcsin_{s}$ as an incomplete beta function. Using the formula \eqref{eq Gauss sum} for the Gauss sum, the calculation of the remaining derivatives is straightforward.
		\end{proof}

		Again, we consider the special case that $c_{f}^{+}(n,h) = 0$ for all $n < 0$.
		
		\begin{corollary}\label{corollary derivative additional part}
			Let $f \in H_{1/2,\tilde{\rho}^{*}_{L}}$, and suppose that $c_{f}^{+}(n,h) = 0$ for all $n <0$ and $h \in L'/L$. Then the Fourier expansion of the derivative $\Phi_{\Delta,r}'(f,z)$ of the Borcherds lift on $\H \setminus H_{\Delta,r}^{-}(f)$ is given by the formula from Proposition \ref{proposition Fourier expansion derivative} plus the expression
			\[
			-2i\sqrt{|\Delta|N}\sum_{h \in L'/L}\sum_{n > 0}c_{f}^{-}(n,h)\sum_{\substack{X \in L_{-|\Delta|n,rh} \\ a \neq 0}}\chi_{\Delta}(X)\frac{\mathbf{1}_{X}(z)\sgn(a)}{Q_{X}(z)},
			\]
			where $\mathbf{1}_{X}(z)$ denotes the characteristic function of the bounded component of $\H \setminus c_{X}$.
		\end{corollary}
		
		\begin{proof}
			This can either be proved by similar arguments as in the proof of Corollary \ref{corollary Borcherds lift Fourier expansion additional part}, or by computing the derivative of the expression \eqref{eq Borcherds lift Fourier expansion additional part}.
		\end{proof}
		
		\section{Applications: Modular integrals with rational period functions and Borcherds products of harmonic Maass forms}
		
		For simplicity, we assume in this section that $N$ is square free. Then the cusps of $\Gamma_{0}(N)$ can be represented by the fractions $1/c$ with $c \mid N$. Note that $\infty$ corresponds to $1/N$. The width of $1/c$ is given by $\alpha_{1/c} = N/c$. We choose the matrix $\sigma_{1/c} \in \SL_{2}(\Z)$ sending $\infty$ to $1/c$ in the form
		\[
		\sigma_{1/c} = \begin{pmatrix}1 & \beta \\ c & N\gamma/c\end{pmatrix}
		\]
		where $\beta,\gamma \in \Z$ are such that $N\gamma/c - c \beta = 1$. Then we can take the Atkin-Lehner involution corresponding to $N/c$ as
		\[
		W_{N/c} = \sigma_{1/c}\begin{pmatrix}N/c & 0 \\ 0 & 1 \end{pmatrix}.
		\]
		We see that $W_{N/c}\infty = 1/c$, so the Atkin-Lehner involutions act transitively on the cusps. Further, the expansion at the cusp $1/c$ of a function $F$, which is modular of weight $k \in \Z$, is given by
		\[
		(F|_{k}\sigma_{1/c})(z) = (c/N)^{k/2} \cdot (F|_{k}W_{N/c})(cz/N).
		\]
		Since 
		\[
		\Phi_{\Delta,r}(f,z)|_{0}W_{N/c} = \Phi_{\Delta,r}(f^{w_{N/c}},z)
		\]
		and consequently
		\[
		\Phi'_{\Delta,r}(f,z)|_{2}W_{N/c} = \Phi'_{\Delta,r}(f^{w_{N/c}},z),
		\]
		the expansion of $\Phi'_{\Delta,r}(f,z)$ at the cusp $1/c$ is essentially given by $\Phi'_{\Delta,r}(f^{w_{N/c}},z)$.
			
		\subsection{Modular integrals with rational period functions}	\label{section modular integrals}
				
		As an application of our extension of the Borcherds lift, we construct modular integrals of weight $2$ for $\Gamma_{0}(N)$ with rational period functions from harmonic Maass forms of weight $1/2$. Following Knopp \cite{knoppeichler}, we call a holomorphic function $F : \H \to \C$ a modular integral of weight $k \in \Z$ for $\Gamma_{0}(N)$ with rational period functions if
		\[
		q_{M}(z) = F(z) - (F|_{k}M)(z)
		\]
		is a rational function of $z$ for each $M \in \Gamma_{0}(N)$, and if $F$ is holomorphic at the cusps of $\Gamma_{0}(N)$, in the sense that $\lim_{y \to \infty}(F|_{k}M)(z)$ exists for every $M \in \SL_{2}(\Z)$. Then the map $M \mapsto q_{M}$ defines a weight $k$ cocycle for $\Gamma_{0}(N)$ with values in the rational functions which are holomorphic on $\H$, i.e., it satisfies
		\[
		q_{MM'} = q_{M}|_{k}M' + q_{M'}
		\]
		for all $M,M' \in \Gamma_{0}(N)$. Conversely, it follows from a more general result of Knopp \cite{knoppeichler} that every such cocycle admits a holomorphic modular integral. Knopp's modular integrals are Poincar\'e series built from the cocycles. It was shown in \cite{ditrational} and \cite{dit} that certain generating series of (traces of) cycle integrals of weakly holomorphic modular functions for $\SL_{2}(\Z)$ are modular integrals of weight $2$ with rational period functions. Using the Borcherds lift we generalize their construction to higher level.

		\begin{proposition}\label{proposition weight 2 modular integral}
			Let $\Delta \neq 1$ be a fundamental discriminant. Let $f \in H_{1/2,\tilde{\rho}^{*}_{L}}$ with $c_{f}^{+}(n,h) = 0$ for all $n <0$ and $h \in L'/L$. Further, assume that $c_{f}^{-}(|\Delta| m^{2}/4N,rm) = 0$ for all $m \in \Z, m > 0$. Then the function
			\[
			F_{\Delta,r}(f,z) = -\frac{1}{4\pi}L_{\Delta}(1)c_{f}^{-}(0,0) + \frac{\bar{\varepsilon}}{\sqrt{N}}\sum_{n = 1}^{\infty}\left(\sum_{d \mid n}\left( \frac{\Delta}{n/d}\right)d \, c_{f}^{+}\left(|\Delta| d^{2}/4N,rd\right)\right)e(nz)
			\]
			is holomorphic on $\H$ and at the cusps of $\Gamma_{0}(N)$, and satisfies the transformation rule
			\[
			 F_{\Delta,r}(f,z)|_{2}M-F_{\Delta,r}(f,z)= -\frac{1}{\pi}\sum_{h \in L'/L}\sum_{n > 0}c_{f}^{-}(n,h)\sum_{\substack{X \in L_{-|\Delta| n,rh} \\ a_{MX} <0 < a_{X}}}\frac{\chi_{\Delta}(X)}{Q_{X}(z)}
			\]
			for all $M \in \Gamma_{0}(N)$, where $a_{X}$ denotes the $a$ entry of $X$. In particular, $F_{\Delta,r}(f,z)$ is a modular integral of weight $2$ for $\Gamma_{0}(N)$.
		\end{proposition}

		\begin{remark}
			\begin{enumerate}
				\item The requirement $c_{f}^{-}(|\Delta| m^{2}/4N,rm) = 0$ for all $m \in \Z, m > 0,$ ensures that $\Phi'_{\Delta,r}(f,z)$ does not have singularities along vertical geodesics, and implies that the second line of the Fourier expansion in Proposition \ref{proposition Fourier expansion derivative} vanishes.
				\item The proof of the transformation behaviour works for arbitrary positive integers $N$, but the assumption that $N$ is square free is used to obtain the Fourier expansions of $\Phi'_{\Delta,r}(f,z)$ at different cusps via Atkin-Lehner operators. One could compute the expansion at a cusp $\ell$ by choosing an appropriate sublattice $K_{\ell}$ instead of $K$ in Proposition \ref{proposition Fourier expansion} and modify the computation of the expansion at $\infty$ correspondingly. However, the above result is certainly true without the assumption that $N$ is square free, but the computations become much more technical.
			\end{enumerate}
		\end{remark}
		
		\begin{proof}
		[Proof of Proposition \ref{proposition weight 2 modular integral}]
			Let $z \in \H \setminus H_{\Delta,r}^{-}(h)$, and let
			\[
			F^{*}_{\Delta,r}(f,z) =  -\frac{1}{2\pi}\sum_{h \in L'/L}\sum_{n > 0}c_{f}^{-}(n,h)\sum_{\substack{X \in L_{-|\Delta| n,rh} \\ a \neq 0}}\chi_{\Delta}(X)\frac{\mathbf{1}_{X}(z)\sgn(a)}{Q_{X}(z)}.
			\]
			By Corollary \ref{corollary derivative additional part} we have 
			\[
			\Phi'_{\Delta,r}(f,z) = 4\pi i \sqrt{N|\Delta|}(F_{\Delta,r}(f,z) + F^{*}_{\Delta,r}(f,z)).
			\]
			Since $\Phi'_{\Delta,r}(f,z)$ transforms like a modular form of weight $2$ for $\Gamma_{0}(N)$, we obtain
			\begin{align*}
			F_{\Delta,r}(f,z)|_{2}M - F_{\Delta,r}(f,z)| = -F^{*}_{\Delta,r}(f,z)|_{2}M + F^{*}_{\Delta,r}(f,z).
			\end{align*}
			Using $Q_{X}(z)|_{-2}M = Q_{M^{-1}X}(z)$, we obtain that the right-hand side of the last formula equals
			\begin{align*}
			 &-\frac{1}{2\pi}\sum_{h \in L'/L}\sum_{n > 0}c_{f}^{-}(n,h)\sum_{\substack{X \in L_{-|\Delta| n,rh} \\ a \neq 0}}\chi_{\Delta}(X)\frac{\mathbf{1}_{X}(z)\sgn(a_{X})-\mathbf{1}_{MX}(Mz)\sgn(a_{MX})}{Q_{X}(z)}.
			\end{align*}
			The characteristic functions $\mathbf{1}_{X}$ and $\mathbf{1}_{MX}$ are related by 
			\[
			\mathbf{1}_{MX}(Mz) = \begin{cases}
			\mathbf{1}_{X}(z), & \text{if } a_{X}\cdot a_{MX} > 0, \\
			1- \mathbf{1}_{X}(z),& \text{if } a_{X}\cdot a_{MX} < 0.
			\end{cases}
			\]
			In particular, all summands with $a_{X}\cdot a_{MX} > 0$ cancel out. In the remaining sum over $X$ with $a_{X}\cdot a_{MX} < 0$, we replace $X$ with $-X$ if $a_{X} < 0$, giving a factor $2$. This proves the transformation behaviour of $F_{\Delta,r}(f,z)$ for $z \in \H \setminus H_{\Delta,r}^{-}(f)$. Since all the functions appearing in the transformation formula are holomorphic on $\H$, we obtain the transformation law by analytic continuation.
			
			Using $\Phi'_{\Delta,r}(f,z)|_{2}W_{d} = \Phi'_{\Delta,r}(f^{w_{d}},z)$ we obtain
			\[
			F_{\Delta,r}(f,z)|_{2}W_{d} = F_{\Delta,r}(f^{w_{d}},z)+F^{*}_{\Delta,r}(f^{w_{d}},z)+F^{*}_{\Delta,r}(f,z)|W_{d}.
			\]
			Since $F_{\Delta,r}(f^{w_{d}},z)$ is holomorphic at $\infty$, and $F^{*}_{\Delta,r}(f^{w_{d}},z)$ and $F^{*}_{\Delta,r}(f,z)|_{2}W_{d}$ vanish as $y \to \infty$, we see that $F_{\Delta,r}(f,z)$ is holomorphic at the cusps.
		\end{proof}
		
		\begin{example}\label{example weight 2 modular integral}
			Let $\Delta > 1$. We apply Proposition \ref{proposition weight 2 modular integral} to a harmonic Maass form $f \in H_{1/2,\rho^{*}_{L}}$ arising as the image of the regularized theta lift studied by Bruinier, Funke and Imamoglu in \cite{bif} of a harmonic Maass form $F \in H_{0}^{+}(\Gamma_{0}(N))$. We assume that the constant coefficients $a_{\ell}^{+}(0)$ of $F$ vanish at all cusps. By Theorem 4.1 in \cite{bif} the Fourier expansion of the $h$-th component of $f$ is given by
			\begin{align*}
			f_{h}(\tau) &= -2\tr_{F}(0,h)\sqrt{v} \\
			&\quad + \sum_{\substack{n < 0}}\tr_{F}(-n,h)\sqrt{v}\beta_{1/2}(4\pi|n|v)e(n\tau) 
 \\
 			&\quad + \sum_{\substack{n > 0}}\frac{\sqrt{N}}{\pi}\tr_{F}(-n,h)e(n\tau) \\
			&\quad + \sum_{\substack{n > 0}}\tr_{F}^{c}(-n^{2}/4N,h) \sqrt{v}\beta_{1/2}^{c}(-4\pi n^{2}v/4N)e(n^{2}\tau/4N),
			\end{align*}
			with the traces
			\[
			\tr_{F}(-n,h) = \begin{dcases}
			\sum_{X \in \Gamma_{0}(N)\backslash L_{-n,h}}\frac{1}{|\overline{\Gamma}_{X}|}F(z_{X}),  &n < 0, \\
			-\delta_{0,h}\frac{1}{2\pi}\int_{\Gamma_{0}(N)\setminus \H}^{\reg}F(z) \frac{dx \, dy}{y^{2}}, & n = 0, \\
			\sum_{X \in \Gamma_{0}(N)\backslash L_{-n,h}}\int_{\Gamma_{X}\backslash c_{X}}^{\reg}F(z)\frac{dz}{Q(z,1)},  &n > 0,
			\end{dcases}
			\]
			and the so-called complementary trace $\tr_{F}^{c}(-n^{2}/4N,h)$, which is defined in \cite{bif}, Section 3. Our definition of the traces of cycle integrals equals $\pi/\sqrt{N}$ times the traces of cycle integrals defined in \cite{bif}, and the traces for $|n|/N$ being a square need to be regularized as explained in \cite{bif}, Section 3. Note that the trace of index $0$ and the complementary trace can be evaluated explicitly in terms of the principal parts of $F$ at the cusps of $\Gamma_{0}(N)$, see \cite{bruinierfunke06}, Remark 4.9, and that the complementary trace is nonzero only for finitely many $n$, see \cite{bruinierfunke06}, Proposition 4.7. Observe that $c_{f}^{+}(n,h) = 0$ for $n < 0$ and $c_{f}^{-}(\Delta m^{2}/4N,rm) = 0$ for $m \in \Z, m > 0,$ if $\Delta > 1$. The Duke-Imamoglu-T\'oth harmonic Maass form $h(\tau)=h_{1}(\tau)$ from the introduction can be constructed as the Bruinier-Funke-Imamoglu lift of $\frac{1}{2}J$.
			
			By Proposition \ref{proposition weight 2 modular integral}, for $\Delta > 1$ a fundamental discriminant the function 
			\[
			F_{\Delta,r}(f,z) = \frac{1}{2\pi}L_{\Delta}(1)\tr_{F}(0,0) + \frac{1}{\pi}\sum_{n = 1}^{\infty}\left(\sum_{d \mid n}\left( \frac{\Delta}{n/d}\right)d\tr_{F}\left(-\Delta d^{2}/4N,rd\right)\right)e(nz) 
			\]
			is a holomorphic function on $\H$, which transforms under the weight $2$ slash operation of $M \in \Gamma_{0}(N)$ by
			\begin{align*}
			F_{\Delta,r}(f,z)|_{2}M - F_{\Delta,r}(f,z) = -\frac{1}{\pi}\sum_{h \in L'/L}\sum_{n > 0}\tr_{F}^{c}(-n^{2}/4N,h)\sum_{\substack{X \in L_{-\Delta n^{2}/4N,rh} \\ a_{MX} < 0 < a_{X}}}\frac{1}{Q_{X}(z)}.
			\end{align*}
			Since $\chi_{\Delta}(X) = 1$ for $X \in L_{-\Delta n^{2}/4N,rh}$ we dropped it from the notation.
			
			In the special case $N = 1$ and $F = \frac{1}{2}J$ (with $\tr_{J}(0,0) = 4$ and $\tr_{J}^{c}(-1/4,1) = 2$) we recover the transformation behaviour of the modular integral $F_{\Delta}(z)$ of Duke, Imamoglu and T\'oth \cite{dit} stated in the introduction.
		\end{example}

	\subsection{Borcherds products}\label{section Borcherds products}

	In this section we construct twisted Borcherds products of harmonic Maass forms $f \in H_{1/2,\tilde{\rho}^{*}_{L}}$. For simplicity we assume $\Delta \neq 1$.

In order to generalize the Borcherds product to the full space $H_{1/2,\tilde{\rho}^{*}_{L}}$ we first recall the construction of certain weight $0$ and weight $2$ cocycles from \cite{ditlinking}, which will appear in the transformation rule of the Borcherds product.

\begin{lemma}
	Let $n > 0$ such that $N|\Delta|n$ is not a square, and let $\mathcal{A} \in \Gamma_{0}(N)\backslash L_{-|\Delta|n,rh}$. Then the function
	\[
	q_{M}^{\mathcal{A}}(z) = \sum_{\substack{X \in \mathcal{A} \\ a_{MX} < 0 < a_{X}}}\frac{1}{Q_{X}(z)}
	\]
	defines a weight $2$ cocycle with values in the rational functions which are holomorphic on $\H$.
\end{lemma}

\begin{proof}
	As in the proof of Proposition \ref{proposition weight 2 modular integral} we compute
	\[
	\sum_{\substack{X \in \mathcal{A} \\ a_{MX} < 0 < a_{X}}}\frac{1}{Q_{X}(z)} =  \sum_{\substack{X \in \mathcal{A} \\ a > 0}}\frac{\mathbf{1}_{X}(z)\sgn(a)}{Q_{X}(z)}-\sum_{\substack{X \in \mathcal{A} \\ a > 0}}\frac{\mathbf{1}_{X}(z)\sgn(a)}{Q_{X}(z)}\bigg|_{2} M
	\]
	for $z$ not lying on any geodesic $c_{X}$ with $X \in \mathcal{A}$. This easily implies that the map $M \mapsto q_{M}^{\mathcal{A}}$ is a weight $2$ cocycle.
\end{proof}

Next, we would like to construct a weight $0$ cocycle $R_{M}^{\mathcal{A}}(z)$ with values in the holomorphic functions on $\H$ such that $\frac{\partial}{\partial z}R_{M}^{\mathcal{A}}(z) = q_{M}^{\mathcal{A}}(z)$. The following proposition gives such a construction for general cocycles with values in rational functions which are holomorphic on $\H$.

\begin{proposition}[\cite{ditlinking}, Theorem 2.1]\label{proposition weight 0 cocycle}
	Let $F(z) = \sum_{n\geq 0}a(n)e(nz)$ be a holomorphic modular integral of weight $2$ for $\Gamma_{0}(N)$ with rational period functions $q_{M} = F|_{2}M-F$. Assume that $a(n) \ll n^{\alpha}$ for some $\alpha > 0$. For $M = \left( \begin{smallmatrix}a & b \\c & d\end{smallmatrix}\right) \in \Gamma_{0}(N)$ with $c \neq 0$ we let
	\begin{align*}
	\Lambda\left(s,\frac{a}{c}\right) = \left(\frac{2\pi}{c}\right)^{-s}\Gamma(s)\sum_{n \geq 1}a(n)e\left(\frac{an}{c}\right)n^{-s}
	\end{align*} 
	and
	\begin{align*}
	H\left(s,\frac{a}{c}\right) &= \Lambda\left(s,\frac{a}{c}\right) + \int_{1}^{\infty}q_{M}(-d/c+it/c)t^{1-s}dt + \frac{a(0)}{s}-\frac{a(0)}{2-s}.
	\end{align*}
	Then $H\left(s,\frac{a}{c}\right)$ is entire and satisfies the functional equation $H\left(s,\frac{a}{c}\right) = -H\left(2-s,-\frac{d}{c}\right)$. Further, for $c \neq 0$ we set
	\[
	R_{M}(z) = -\frac{i}{c}H\left(1,\frac{a}{c}\right) + \int_{-\frac{d}{c} + \frac{i}{c}}^{z}q_{M}(w)dw + a(0) \frac{a+d}{c},
	\]
	and for $M = \pm \left(\begin{smallmatrix}1 & n \\ 0 &1 \end{smallmatrix}\right)$ we let $R_{M}(z) = na(0)$. Then $R_{M}(z)$	defines a weight $0$ cocycle for $\Gamma_{0}(N)$ with values in the holomorphic functions on $\H$, and which satisfies $\frac{\partial}{\partial z}R_{M}(z) = q_{M}(z)$ for every $M \in \Gamma_{0}(N)$. 
\end{proposition}

\begin{proof}
	The proof is exactly the same as that of \cite{ditlinking}, Theorem 2.1, so we only give a sketch. By a standard computation we obtain for $c \neq 0$ the integral representation
	\[
	H\left(s,\frac{a}{c} \right) = -\int_{1}^{\infty}(F(z_{1/t})-a(0))t^{1-s}dt + \int_{1}^{\infty}(F(M z_{t})-a(0))t^{s-1}dt,
	\]
	where $z_{t} = -\frac{d}{c}+\frac{i}{ct}$. Since $z_{1/t} = -\frac{d}{c}+\frac{it}{c}$ and $M z_{t} = \frac{a}{c}+\frac{it}{c}$, we see that $H\left(s,\frac{a}{c} \right)$ is entire and satisfies the claimed functional equation. Further, we let 
	\[
	G(z) = a(0)z + \sum_{n \geq 1}\frac{a(n)}{2\pi i n}e(nz)
	\]
	be a primitive of $F(z)$. By taking the limit $s \to 1$ in $H\left(s,\frac{a}{c} \right)$ we obtain after a short calculation
	\[
	R_{M}(z) = G(M z) - G(z),
	\]
	which is valid for all $M \in \Gamma_{0}(N)$ and defines a weight $0$ cocycle with values in the holomorphic functions on $\H$, and $\frac{\partial}{\partial z}R_{M}(z) = q_{M}(z)$.
\end{proof}

\begin{lemma}
	Let $q_{M}^{\mathcal{A}}$ be the weight $2$ cocycle associated to $\mathcal{A} \in \Gamma_{0}(N) \setminus L_{-|\Delta|n,rh}$ as above. For $X \in \mathcal{A}$ let $w_{X} > w_{X}'$ denote the two real endpoints of the geodesic $c_{X}$. Let $F(z) = \sum_{n \geq 0}a(n)q^{n}$ be a modular integral for $q_{M}^{\mathcal{A}}$ with $a(n) \ll n^{\alpha}$ for some $\alpha > 0$ and let $M = \left(\begin{smallmatrix}a & b \\ c & d \end{smallmatrix} \right) \in \Gamma_{0}(N)$. Further, for $c \neq 0$ let
	\begin{align*}
	L_{F}\left(s,\frac{a}{c}\right) = \sum_{n \geq 1}a(n)e\left(\frac{an}{c}\right)n^{-s}.
	\end{align*}
	and
	\[
	R_{M}^{\mathcal{A}}(z) = \frac{1}{\sqrt{4N|\Delta|n}}\sum_{\substack{X \in \mathcal{A} \\ a_{MX} < 0 < a_{X}}}\left(\log(z-w_{X}) - \log(z-w_{X}')\right) + \frac{1}{2\pi i }L_{F}\left(1,\frac{a}{c}\right) + a(0)\frac{a+d}{c},
	\]
	and for $M = \pm \left(\begin{smallmatrix}1 & n \\ 0 &1 \end{smallmatrix}\right)$ we let $R_{M}^{\mathcal{A}}(z) = na(0)$. Then $R_{M}^{\mathcal{A}}(z)$ is a weight $0$ cocycle with values in the holomorphic functions on $\H$ which satisfies $\frac{\partial}{\partial z}R_{M}^{\mathcal{A}}(z) = q_{M}^{\mathcal{A}}(z)$.
\end{lemma}

\begin{proof}
	Note that
	\[
	q_{M}^{\mathcal{A}}(z) = \frac{1}{\sqrt{4N|\Delta|n}}\sum_{\substack{X \in \mathcal{A} \\ a_{MX} < 0 < a_{X}}}\left(\frac{1}{z-w_{X}} - \frac{1}{z-w_{X}'}\right).
	\]
	Thus if we choose
	\[
	\frac{1}{\sqrt{4N|\Delta|n}}\sum_{\substack{X \in \mathcal{A} \\ a_{MX} < 0 < a_{X}}}\left(\log(z-w_{X}) - \log(z-w_{X}')\right)
	\]
	as a primitive for $q_{M}^{\mathcal{A}}(z)$, the formula for $R_{M}^{\mathcal{A}}(z)$ follows from Proposition \ref{proposition weight 0 cocycle}.
\end{proof}

\begin{example}
	Let $N = 1, \Delta > 1,$ and $M = S = \left(\begin{smallmatrix}0 & -1 \\ 1 & 0 \end{smallmatrix}\right)$. We have
	\[
	q_{S}^{\mathcal{A}}(z) = \sum_{\substack{X \in \mathcal{A} \\ c < 0 < a}}\frac{1}{Q_{X}(z)}.
	\]
	It easily follows from the definition and the functional equation of $H(s,0)$ given in Proposition \ref{proposition weight 0 cocycle} that 
	\[
	L_{F}(1,0) = -\frac{2\pi i }{\sqrt{4\Delta n}} \sum_{\substack{X \in \mathcal{A} \\ c < 0 < a}}\left(\log(i-w_{X})-\log(i-w_{X}')\right)
	\]
	independently of the modular integral $F$ for $q^{\mathcal{A}}$. In particular, we obtain
	\[
	R_{S}^{\mathcal{A}}(z) = \frac{1}{\sqrt{4\Delta n}}\sum_{\substack{X \in \mathcal{A} \\ c < 0 < a}}\left(\log\left(\frac{z-w_{X}}{i-w_{X}}\right)-\log\left(\frac{z-w_{X}'}{i-w_{X}'}\right)\right).
	\]
\end{example}

We can now state the transformation behaviour of the Borcherds product associated to $f \in H_{1/2,\tilde{\rho}^{*}_{L}}$.

\begin{theorem}\label{theorem new borcherds products}
	Let $\Delta \neq 1$ be a fundamental discriminant. Let $f \in H_{1/2,\tilde{\rho}^{*}_{L}}$ and suppose that $c_{f}^{+}(|\Delta|m^{2}/4N,rm) \in \R$ for all $m \in \Z, m > 0$. Further, assume that $c_{f}^{+}(n,h) = 0$ for all $n < 0, h \in L'/L,$ and that $c_{f}^{-}(|\Delta|m^{2}/4N,rm) = 0$ for all $m \in \Z,m > 0$. Then the infinite product
	\begin{align*}
	\Psi_{\Delta,r}(f,z) &= e\left( \frac{\sqrt{|\Delta| N}}{4\pi}L_{\Delta}(1)c_{f}^{-}(0,0)z\right)\prod_{m=1}^{\infty}\prod_{b (\Delta)}[1-e(mz+b/\Delta)]^{\left(\frac{\Delta}{b}\right)c_{f}^{+}(|\Delta|m^{2}/4N,rm)}
	\end{align*}
	converges to a holomorphic function on $\H$ transforming as
	\begin{align*}\label{eq Borcherds product transformation}
	\Psi_{\Delta,r}(f,Mz) = \chi(M)\mu_{\Delta,r}(f,M,z)\Psi_{\Delta,r}(f,z)
	\end{align*}
	for all $M \in \Gamma_{0}(N)$, where $\chi$ is a character of $\Gamma_{0}(N)$ and
	\[
	\mu_{\Delta,r}(f,M,z) = \prod_{h \in L'/L}\prod_{n > 0}\prod_{\mathcal{A} \in \Gamma_{0}(N)\backslash L_{-|\Delta|n,rh}}e\left(-\frac{\sqrt{|\Delta|N}}{\pi}c_{f}^{-}(n,h)\chi_{\Delta}(\mathcal{A})R_{M}^{\mathcal{A}}(z)\right),
	\]
	where $R_{M}^{\mathcal{A}}(z)$ is the weight $0$ cocycle with $\frac{\partial}{\partial z}R_{M}^{\mathcal{A}}(z) = q_{M}^{\mathcal{A}}(z)$. Further, its logarithmic derivative is given by
	\[
	\frac{\partial}{\partial z}\log\left(\Psi_{\Delta,r}(f,z)\right) = -2\pi i\sqrt{|\Delta|N} F_{\Delta,r}(f,z),
	\]
	where $F_{\Delta,r}(f,z)$ is the modular integral defined in Proposition \ref{proposition weight 2 modular integral}. 
\end{theorem}

\begin{proof}
	Using Proposition \ref{proposition weight 2 modular integral} we see after a short calculation that the logarithmic derivatives of $\Psi_{\Delta,r}(f,Mz)$ and $\mu_{\Delta,r}(f,M,z)\Psi_{\Delta,r}(f,z)$ agree. Further, both functions are holomorphic and non-vanishing on $\H$. Hence they are constant multiples of each other. This proves the transformation behaviour. 
	
	The fact that $R_{M}^{\mathcal{A}}(z)$ is a weight $0$ cocycle together with the transformation formula of the Borcherds product implies that $\chi$ is a character of $\Gamma_{0}(N)$.
\end{proof}

\begin{example}
	Let $\Delta > 1$, and let $f \in H_{1/2,\rho^{*}_{L}}$ be the Bruinier-Funke-Imamoglu lift of a harmonic Maass form $F \in H_{0}^{+}(\Gamma_{0}(N))$ with vanishing constant coefficients $a_{\ell}^{+}(0)$ at all cusps as in Example~\ref{example weight 2 modular integral}. Its Borcherds lift is given by
	\begin{align*}
	\Psi_{\Delta,r}\left(\frac{\pi}{\sqrt{N}}f,z\right) &= \prod_{m=1}^{\infty}\prod_{b (\Delta)}[1-e(mz+b/\Delta)]^{\left(\frac{\Delta}{b}\right)\tr_{F}(-\Delta m^{2}/4N,rm)} \\
	&\quad\times e\left( -\frac{\sqrt{\Delta}}{2}L_{\Delta}(1)\tr_{F}(0,0)z\right).
	\end{align*}
	For $N = 1$ and $F = J =j -744$ (with $\tr_{J}(0,0) = 4$ and $\tr_{J}^{c}(-1/4,1) = 2$) we obtain the theorem in the introduction. Note that the relations $S^{4} = 1, (ST)^{6} = 1$ and $\chi(T) = 1$ imply that $\chi = 1$ for $N = 1$.
\end{example}

\bibliography{references}{}
\bibliographystyle{plain}

\end{document}